\documentclass[12pt,english]{article}
\usepackage[T1]{fontenc}
\usepackage[latin9]{inputenc}
\usepackage{geometry}
\usepackage{amsmath}
\usepackage{amsthm}
\usepackage{amssymb}
\usepackage{hyperref}
\usepackage{comment}
\hypersetup{
     colorlinks   = true,
     citecolor    = blue,
     linkcolor    = blue
}

\makeatletter
\numberwithin{equation}{section}
\numberwithin{figure}{section}

\theoremstyle{definition}
\newtheorem{defn}{\protect\definitionname}[section]
\theoremstyle{remark}
\newtheorem{rem}{\protect\remarkname}[section]
\theoremstyle{definition}
\newtheorem{example}{\protect\examplename}[section]
\theoremstyle{plain}
\newtheorem{prop}{\protect\propositionname}[section]
\theoremstyle{plain}
\newtheorem{conjecture}{\protect\conjecturename}[section]
\theoremstyle{plain}
\newtheorem{thm}{\protect\theoremname}[section]
\theoremstyle{plain}
\newtheorem{lem}{\protect\lemmaname}[section]
\theoremstyle{plain}
\newtheorem{cor}{\protect\corollaryname}[section]

\usepackage{graphicx}

\makeatother

\usepackage{babel}
\providecommand{\conjecturename}{Conjecture}
\providecommand{\definitionname}{Definition}
\providecommand{\examplename}{Example}
\providecommand{\lemmaname}{Lemma}
\providecommand{\propositionname}{Proposition}
\providecommand{\remarkname}{Remark}
\providecommand{\theoremname}{Theorem}
\providecommand{\corollaryname}{Corollary}

\begin{document}
\title{Path Developments and Tail Asymptotics of Signature for Pure Rough
Paths}
\author{Horatio Boedihardjo\thanks{Department of Mathematics and Statistics, University of Reading,
Reading RG6 6AX, United Kingdom. Email: h.s.boedihardjo@reading.ac.uk.}, Xi Geng\thanks{School of Mathematics and Statistics, University of Melbourne, Melbourne VIC 3010, Australia. Email: xi.geng@unimelb.edu.au.} and Nikolaos P. Souris\thanks{Department of Mathematics and Statistics, University of Reading,
Reading RG6 6AX, United Kingdom. Email: n.souris@reading.ac.uk.}}
\date{}
\maketitle
\begin{abstract}
Solutions to linear controlled differential equations can be
expressed in terms of iterated path integrals of the driving path.
This collection of iterated integrals encodes essentially all information
about the driving path. While upper bounds for iterated path integrals are
well known, lower bounds are much less understood, and it is
known only relatively recently that some type of asymptotics for the $n$-th
order iterated integral can be used to recover some
intrinsic quantitative properties of the path, such as the length of $C^1$ paths.

In the present paper, we investigate the simplest type of rough paths
(the rough path analogue of line segments), and establish uniform
upper and lower estimates for the tail asymptotics of iterated integrals
in terms of the local variation of the path. Our methodology, which
we believe is new for this problem, involves developing paths into
complex semisimple Lie algebras and using the associated representation
theory to study spectral properties of Lie polynomials under the Lie
algebraic development.
\end{abstract}

\newpage
\tableofcontents

\section{Introduction}

Controlled differential equations of the form 
\begin{equation}
dY_{t}=\sum_{i=1}^{d}V_{i}(Y_{t})dX_{t}^{i}\label{eq: ContEqn}
\end{equation}
where $V_{i}:\mathbb{R}^{N}\rightarrow\mathbb{R}^{N}$, $X:[0,T]\rightarrow\mathbb{R}^{d},$
$Y:[0,T]\rightarrow\mathbb{R}^{N}$, frequently appear in many interesting
problems in stochastic analysis and applications to stochastic modelling (cf. \cite{CJY09}, \cite{Kallianpur80}, \cite{Oksendal13}, \cite{YZ99} and the references therein). The most well known and fundamental
example is perhaps when $X_{t}$ is a Brownian motion. The rough path
theory initiated by Lyons \cite{Lyons98} and further developed by
many authors (cf. \cite{Davie07}, \cite{Gubinelli04}, \cite{Gubinelli10}), identifies a wide class of ``rough'' paths including Brownian motion for which
the equation (\ref{eq: ContEqn}) is well defined. The theory is analytically
consistent with the classical viewpoint, in the sense that it is
a continuous extension of the Lebesgue-Stieltjes theory with respect
to the rough path topology and reduces to the classical setting when
the underlying paths have finite lengths. Rough path theory naturally motivates the
study of analytic properties of solutions to (\ref{eq: ContEqn})
driven by rough paths.

One particularly tractable class of examples is when the vector fields
$(V_{i})_{1\leqslant i\leqslant d}$ are linear. In this case, the
solution at time $t=T$ can be represented explicitly as 
\[
Y_{T}=\sum_{n=0}^{\infty}\sum_{i_{1},\cdots,i_{n}=1}^{d}V_{i_{1}}\cdots V_{i_{n}}(Y_{0})\cdot\int_{0<t_{1}<\cdots<t_{n}<T}dX_{t_{1}}^{i_{i}}\cdots dX_{t_{n}}^{i_{n}}.
\]
In particular, $Y_{T}$ depends on the driving path $X$ through the
collection of iterated coordinate integrals 
\[
S(X)\triangleq\left\{ \int_{0<t_{1}<\cdots<t_{n}<T}dX_{t_{1}}^{i_{1}}\cdots dX_{t_{n}}^{i_{n}}:\ n\geqslant1,\ 1\leqslant i_{1},\cdots,i_{n}\leqslant d\right\} .
\]
For algebraic reasons, it is useful to think of this collection as
a single element of the infinite tensor algebra $T((\mathbb{R}^d))\triangleq\prod_{n=0}^\infty (\mathbb{R}^d)^{\otimes n}$,
more intrinsically as 
\[
S(X)=1+\sum_{n=1}^{\infty}\int_{0<t_{1}<\cdots<t_{n}<T}dX_{t_{1}}\otimes\cdots\otimes dX_{t_{n}}.
\]
This tensor element $S(X)$, known as the \textit{signature} of the path $X$,
plays an essential role in rough path theory. The significance and
usefulness of path signature is based on a fundamental theorem which asserts that every
(weakly geometric) rough path is uniquely determined by its signature
up to tree-like pieces (cf. \cite{HL10} and \cite{BGLY16}). However, the proof of this uniqueness result is non-constructive
and does not contain information about how one can reconstruct
a rough path from its signature. The general reconstruction problem was studied
by many authors (cf. \cite{LX18}, \cite{Geng17}, \cite{Chang18}).\\

On the other hand, combining with algebraic properties of signature,
the uniqueness result ensures that essentially all information about
the rough path is encoded in the \textit{tail} of its signature, i.e. when
looking at the component $\int dX_{t_{1}}\otimes\cdots\otimes dX_{t_{n}}$
in the asymptotics as $n\rightarrow\infty.$ An interesting
question arises naturally as follows. \\
\\\textit{Question: Are there explicit and elegant formulae allowing
us to recover intrinsic quantities of the path from its signature
tail asymptotics?}\\

The study of this question begins by observing the following elementary estimate

\begin{equation}
\left\Vert \int_{0<t_{1}<\cdots<t_{n}<T}dX_{t_{1}}\otimes\cdots\otimes dX_{t_{n}}\right\Vert \leqslant\frac{\|X\|_{1\text{-var}}^{n}}{n!}\label{eq: BVFactEst}
\end{equation}
when $X$ has finite length. A surprising and highly non-trivial fact is
that this simple estimate becomes asymptotically sharp as $n\rightarrow\infty,$
at least for the class of $C^{1}$ paths. In a precise and elegant
way, it was shown by Hambly-Lyons \cite{HL10}, and subsequently
by Lyons-Xu \cite{LX15} that the tail asymptotics of the normalized
signature recovers the length of a $C^{1}$ path with unit speed parametrization:

\begin{equation}
\lim_{n\rightarrow\infty}\left(n!\left\Vert \int_{0<t_{1}<\cdots<t_{n}<T}dX_{t_{1}}\otimes\cdots\otimes dX_{t_{n}}\right\Vert \right)^{\frac{1}{n}}=\|X\|_{1\text{-var}}.\label{eq: BVLengthConj}
\end{equation}
Whether the same formula is true for general paths with finite length
remains an important and challenging open problem.

The rough path analogue of the factorial estimate (\ref{eq: BVFactEst})
becomes much deeper, and the following type of uniform upper estimate for rough
paths with finite $p$-variation was due to Lyons \cite{Lyons98}
(cf. Theorem \ref{thm: LyonsExt} below):
\[
\left\Vert \int_{0<t_{1}<\cdots<t_{n}<T}dX_{t_{1}}\otimes\cdots\otimes dX_{t_{n}}\right\Vert \leqslant\frac{C_{p}\cdot \|\mathbf{X}\|^n_{p\text{-var}}}{(n/p)!}. \label{eq:FactorialDecayRough}
\]
If one
believes that the above estimate is asymptotically sharp as $n\rightarrow\infty$ for paths whose intrinsic roughness is $p$, we are naturally led to considering
the quantity
\begin{equation}
L_{p}({\bf X})\triangleq\limsup_{n\rightarrow\infty}\left(\left(\frac{n}{p}\right)!\left\Vert \int_{0<t_{1}<\cdots<t_{n}<T}dX_{t_{1}}\otimes\cdots\otimes dX_{t_{n}}\right\Vert \right)^{\frac{p}{n}}\label{eq: SigTailIntro}
\end{equation}
constructed from the tail of signature, and looking
for its connection with intrinsic
properties of the path ${\bf X}$. The quantity $L_p(\mathbf{X})$ certainly does not recover the usual
$p$-variation since $L_{p}({\bf X})=0$ for any bounded variation
path (see \ref{eq: BVFactEst}). The first hint about the meaning of $L_{p}({\bf X})$ was provided
by Boedihardjo and Geng \cite{BG17}, in which the authors showed that,
when ${\bf X}$ is a Brownian motion and $p=2$, with probability
one $L_{p}({\bf X})$ is a deterministic constant multiple of the
quadratic variation of Brownian motion. To some extent, this is  suggesting
that, $L_{p}({\bf X})$ may be intimately related to certain notion of \textit{local} $p$-variation defined in a similar way to the usual $p$-variation but along partitions with arbitrarily fine scales. \\

The main goal of the present paper is to investigate this problem
at a precise quantitative level for the simplest class of deterministic
rough paths resembling line segments. These paths, also known as \textit{pure
rough paths}, are of the form ${\bf X}_{t}=\exp(tl)$ ($0\leqslant t\leqslant1$)
where $l$ is a Lie polynomial of degree $m\geqslant1$. If $m=1,$ ${\bf X}_{t}$
becomes a classical line segment represented by the vector $l\in V$. In general, ${\bf X}_{t}$
carries an intrinsic roughness of $m$.

We are going to show that, for any pure rough path ${\mathbf{X}_t=\exp(tl)}$ over
$\mathbb{R}^{d}$ with roughness $m$, under the projective tensor
norm, the signature tail asymptotics $L_{m}({\bf X})$ defined by
(\ref{eq: SigTailIntro}) with $p=m$ is precisely related to the
highest degree component $l_m$ of the Lie polynomial $l$
 through the estimate 
\begin{equation}
c(m,d)\cdot\| l_m\|\leqslant L_{m}({\bf X})\leqslant\|l_m\|,\label{eq: UpperLower}
\end{equation}
where $c(m,d)\in (0,1]$ is a constant depending only on the roughness
$m$ and the dimension $d$ which also admits an explicit lower estimate.
The quantity $\|l_m\|$ is related to a notion of local $m$-variation of $\mathbf{X}$ 
as seen from Proposition \ref{prop: LocalPVar} below.
When $d=2$ and $m=2,3$, we have $c(m,d)=1$ and therefore
\begin{equation}\label{eq: PureConj}
L_{m}({\bf X})=\|l_m\|.
\end{equation}The same conclusion also holds for some cases in degrees $m=4,5$. The precise
formulation of our main result is given by Theorem \ref{thm:MainThm}
below. On the other hand, if one works with the Hilbert-Schmidt tensor norm, there is
also a class of pure rough paths for which $c(m,d)=1$. We conjecture that the formula (\ref{eq: PureConj}) is true
for arbitrary pure rough paths. 
\begin{comment}
It is aslo interesting to point out that,
$$\|{\bf X}\|_{\text{local-}m\text{-var}}=\|\pi_{m}(l)\|$$ for pure rough paths ${\bf X}_{t}=\exp(tl)$ of roughness $m$,
where $\pi_{m}$ is the canonical projection onto the degree $m$
component of $l$. Therefore, the problem can also be formulated from an independent algebraic viewpoint without mentioning any connection with paths, although it becomes more natural from the path-based motivation.
\end{comment}

Our proof of the upper bound in (\ref{eq: UpperLower}) relies on
combinatorial arguments. The core of our work, which lies in establishing
a matching lower bound, is a novel method based on the representation
theory of complex semisimple Lie algebras. To be more precise, our
starting point is a general representation of the tensor algebra that
allows us to develop paths onto an automorphism group from Cartan's
viewpoint. Specific choices of such representations were already used
by Hambly-Lyons \cite{HL10} and Lyons-Xu \cite{LX15} for proving (\ref{eq: BVLengthConj})
for $C^{1}$ paths, and also by Chevyrev-Lyons \cite{CL16} and Lyons-Sidorova \cite{LS06} for studying other signature-related properties.
The key ingredient in our approach, is to allow such representation
factor through a complex semisimple Lie algebra $\mathfrak{g}$ and develop the highest
degree component of Lie polynomials into a so-called Cartan subalgebra of $\mathfrak{g}$.
It turns out that, under this semisimple picture, the associated representation theory enables us to study spectral
properties of the highest degree Lie component in an effective and quantitative way. We explain the strategy and elaborate these points more precisely in Section \ref{subsec: LowerBdd}
as we develop the mathematical details.

It is also worthwhile to mention that, as an immediate application of our methodology, one can prove a separation of points property for path signatures. More specifically,  if $g_1$ and $g_2$ are two distinct group-like elements as the signatures of two different rough paths over $\mathbb{R}^d$, then one can find a finite dimensional development $\Phi:\mathbb{R}^d\rightarrow\mathrm{End}(W)$ (cf. Definition \ref{def: AlgDev} below) such that $\Phi(g_1)\neq \Phi(g_2)$. The precise formulation and proof of this fact is given in Corollary \ref{cor: SepProp} below. Such a separation property was first obtained by Chevyrev-Lyons \cite{CL16} as the key point of proving their uniqueness result for the expected signature of stochastic processes.\\
\\
\textbf{Organization of the paper.} In Section \ref{sec: Background},
we recall some basic notions from rough path theory and
then formulate our main result in Theorem \ref{thm:MainThm}. In Section
\ref{sec: SpecExam}, we give some heuristics on the underlying problem
by discussing some special examples. Another result that complements
our main result is stated in Theorem \ref{thm: FreeLie}. In Section
\ref{sec: MainProof}, we develop the proof of our main result. Section
\ref{subsec: Upper} is devoted to the upper estimate, and Section \ref{subsec: LowerBdd}
is devoted to the lower estimate in which we divide the proof into
several intermediate steps and results. In Section
\ref{sec: FreeLie}, we give the independent proof of Theorem \ref{thm: FreeLie}.

\section{\label{sec: Background}Notions from rough path theory and statement of main
result}

In this section, we recall some basic ideas and notions from the rough
path theory developed by Lyons \cite{Lyons98}. We refer the reader
to the monographs by Lyons-Qian \cite{LQ02} and Friz-Victoir \cite{FV10} for a systematic introduction.
After that, we formulate the main result of the present paper.

\subsection{The rough path structure}

The fundamental insight of rough path theory is that, beyond certain
level of regularity, the structure encoded in a given path living in some
Banach space $V$ becomes no longer sufficient for yielding an analytically
consistent notion of integration and differential equations, and thus higher
order structures (iterated path integrals) need to be specified along with the underlying path
as a priori information. Mathematically, a rough path should be viewed
as a generic path living inside some tensor group
in which the state space $V$ is embedded as the first order structure.

Let $(V, \| \cdot \|)$ be a given fixed Banach space over $\mathbb{F}=\mathbb{R}$
or $\mathbb{C}$. 

\begin{defn}
A sequence $\{\|\cdot\|_{V^{\otimes_{a}m}}:m\geqslant1\}$
of norms on the algebraic tensor products $\{V^{\otimes_{a}m}:m\geqslant1\}$
are called \textit{reasonable tensor algebra norms} if\\
\\
(i) $\|\cdot\|_{V}=\|\cdot\|$;\\
(ii) $\|\xi\otimes\eta\|_{V^{\otimes_{a}(m+n)}}\leqslant\|\xi\|_{V^{\otimes_{a}m}}\cdot\|\eta\|_{V^{\otimes_{a}n}}$
for $\xi\in V^{\otimes_{a}m}$ and $\eta\in V^{\otimes_{a}n}$;\\
(iii) $\|\phi\otimes\psi\|\leqslant\|\phi\|\cdot\|\psi\|$ for $\phi\in(V^{\otimes_{a}m})^{*}$
and $\psi\in(V^{\otimes_{a}n})^{*}$ where the norms are the induced
dual norms;\\
(iv) $\|P^{\sigma}(\xi)\|_{V^{\otimes_{a}m}}=\|\xi\|_{V^{\otimes_{a}m}}$
for $\xi\in V^{\otimes_{a}m}$ and $\sigma$ being a permutation of
order $m$, where $P^{\sigma}$ is the permutation operator induced by $\sigma$ on $m$-tensors.
\end{defn}

It is known that the inequalities in (ii) and (iii) automatically become
equalities (cf. \cite{DU77}). The completion of $V^{\otimes_{a}m}$ under $\|\cdot\|_{V^{\otimes_{a}m}}$
is denoted as $(V^{\otimes m},\|\cdot\|_{V^{\otimes m}}).$

Examples of reasonable tensor algebra norms include the projective tensor norm, the injective tensor norm, and the Hilbert-Schmidt tensor norm if $V$ is a Hilbert space. Since the projective tensor norm is mostly relevant to us, we recall its definition here. Given $\xi\in V^{\otimes_a m}$, the \textit{projective tensor norm} of $\xi$ is defined by\[
\|\xi\|_{{\rm proj}}\triangleq\inf\left\{ \sum_{i=1}^{r}\|v_{1}^{i}\|\cdots\|v_{m}^{i}\|:\xi=\sum_{i=1}^{r}v_{1}^{i}\otimes\cdots\otimes v_{m}^{i}\ \text{with }r\geqslant1,\!v_{j}^{i}\in V\right\}.
\] Given a fixed norm on $V$, the associated projective tensor norm is the largest among all reasonable tensor algebra norms. It admits the following dual characterization (cf. \cite{Ryan02}): 
\begin{equation}
\|\xi\|_{{\rm proj}}=\sup\left\{ |B(\xi)|:B\in{\cal L}(V\times\cdots\times V;\mathbb{R}),\ \|B\|\leqslant1\right\} .\label{eq: DualProj}
\end{equation}
When $V=\mathbb{R}^d$ is equipped with the $l^1$-norm with respect to the standard basis, the associated projective tensor norm on $V^{\otimes m}$ coincides with the $l^1$-norm with respect to the canonical tensor basis.

From now
on, we assume that a sequence of reasonable tensor algebra norms are given and fixed. We often omit the subscript when the norms are clear from the context.

Let $T((V))$ be the \textit{infinite tensor algebra} consisting
of formal tensor series $\xi=(\xi_{0},\xi_{1},\xi_{2},\cdots)$ with
$\xi_{n}\in V^{\otimes n}$ for each $n$ ($V^{\otimes0}\triangleq\mathbb{F}$). Given $n\geqslant1,$ let $T^{(n)}(V)\triangleq\oplus_{k=0}^{n}V^{\otimes k}$
be the\textit{ truncated tensor algebra} of degree $n$. There are natural notions of exponential and logarithm over these tensor algebras defined by using the standard Taylor expansion formula with respect to the tensor product. For instance, the exponential function over $T((V))$ is given by 
\[
\exp(\xi)\triangleq\sum_{n=0}^{\infty}\frac{1}{n!}\xi^{\otimes n},\ \ \ \xi\in T((V)),
\]while over $T^{(n)}(V)$ it is defined by the same formula but truncated up to degree $n$.

\begin{defn}
\label{def: RoughPath}A \textit{multiplicative functional} of degree
$n$ is a continuous functional 
\[
{\bf X}=(1,X^{1},\cdots,X^{n}):\Delta_{T}\triangleq\{(s,t):0\leqslant s\leqslant t\leqslant T\}\rightarrow T^{(n)}(V)
\]
such that ${\bf X}_{s,u}={\bf X}_{s,t}\otimes{\bf X}_{t,u}$ for all
$s\leqslant t\leqslant u.$ Given a real number $p\geqslant1,$ ${\bf X}$
is said to have \textit{finite total $p$-variation} if
\begin{equation}
\|{\bf X}\|_{p\text{-var}}\triangleq\sum_{k=1}^{n}\sup_{{\cal P}}\left(\sum_{t_{i}\in{\cal P}}\|X_{t_{i-1},t_{i}}^{k}\|^{\frac{p}{k}}\right)^{\frac{k}{p}}<\infty,\label{eq: p-var}
\end{equation}
\noindent where the supremum is taken over all finite partitions of $[0,T]$.
A \textit{rough path} with roughness $p$ (or simply a $p$-rough
path) is a multiplicative functional of degree $\lfloor p\rfloor$
which has finite total $p$-variation, where $\lfloor p\rfloor$ denotes
the largest integer not exceeding $p$.
\end{defn}
\begin{rem}
Due to multiplicativity, a rough path ${\bf X}_{s,t}$ can be equivalently
regarded as an actual path ${\bf X}_{t}\triangleq{\bf X}_{0,t}$ and
vice versa by ${\bf X}_{s,t}\triangleq{\bf X}_{s}^{-1}\otimes{\bf X}_{t}.$
We do not distinguish these two viewpoints.
\end{rem}
The notion of rough paths is mostly useful when a crucial Lie algebraic
property is satisfied. Recall that there is a natural Lie structure
on the tensor algebra given by $[\xi,\eta]\triangleq\xi\otimes\eta-\eta\otimes\xi.$
The space of \textit{homogeneous Lie polynomials} of degree $n$,
denoted as ${\cal L}_{n}(V)$, is the norm completion of the algebraic
space ${\cal L}_{n}^{a}(V)$ defined inductively by ${\cal L}_{1}^{a}(V)\triangleq V$
and ${\cal L}_{n+1}^{a}(V)\triangleq[V,{\cal L}_{n}^{a}(V)].$  
Define the space of \textit{Lie polynomials} of degree $n$ by
\[
{\cal L}^{(n)}(V)\triangleq\oplus_{k=1}^{n}{\cal L}_{k}(V)
\]
and the \textit{free nilpotent group} of degree $n$ by
\[
G^{(n)}(V)\triangleq\exp({\cal L}^{(n)}(V))
\]
respectively. They are both canonically embedded inside $T^{(n)}(V)$.
\begin{defn}
A $p$-rough path is said to be \textit{weakly geometric} if it takes
values in the group $G^{(\lfloor p\rfloor)}(V)$.
\end{defn}
Weakly geometric rough paths cover a wide range of interesting examples,
for instance bounded variation paths ($p=1$), Brownian motion and
continuous semimartingales ($2<p<3$), wide classes of Gaussian processes
and Markov processes etc. This is the appropriate class of paths which the rough path theory of integration and differential equations is based on. 

\subsection{The signature of a rough path}

An important aspect of rough path theory is the characterization
of rough paths in terms of the so-called path signature, which
is a generalized notion of iterated path integrals. Its definition
is based on the following basic property of rough paths proved by
Lyons \cite{Lyons98}. 

\begin{thm}[Lyons' Extension Theorem]\label{thm: LyonsExt}

Let ${\bf X}=({\bf X}_{s,t})_{0\leqslant s\leqslant t\leqslant T}$
be a $p$-rough path. Then there exists a unique extension of ${\bf X}$
to a multiplicative functional $\mathbb{X}:\Delta_T\rightarrow T((V)):$
\[
(s,t)\mapsto\mathbb{X}_{s,t}=(1,X_{s,t}^{1},\cdots,X_{s,t}^{\lfloor p\rfloor},\cdots,X_{s,t}^{n},\cdots),
\]
whose restriction to $T^{(n)}(V)$ has
finite total $p$-variation for all $n\geqslant\lfloor p\rfloor+1.$
Moreover, there exist a universal constant $\beta_{p}$ depending
only on $p$ and a nonnegative function $\omega_\mathbf{X}(s,t)$ related to the $p$-variation of $\mathbf{X}$, such that 
\begin{equation}
\|X_{s,t}^{n}\|\leqslant\frac{\omega_\mathbf{X}(s,t)^{n/p}}{\beta_{p}(n/p)!},\ \ \ {\rm for\ all}\ n\geqslant1\ {\rm and}\ (s,t)\in\Delta_{T},\label{eq: FactEst}
\end{equation}
where the factorial $(n/p)!$ is defined by using the Gamma function.

\end{thm}
\begin{defn}
The tensor series $\mathbb{X}_{0,T}\in T((V))$ is called the \textit{signature}
of ${\bf X}.$ It is usually denoted as $S({\bf X}).$
\end{defn}
\begin{example}
If $(X_{t})_{0\leqslant t\leqslant T}$ is a bounded variation path,
then its signature is precisely the sequence of iterated path integrals
\[
\left(1,X_{T}-X_{0},\int_{0<s<t<T}dX_{s}\otimes dX_{t},\cdots\right)\in T((V))
\]defined in the sense of Lebesgue-Stieltjes. 
In this case, the factorial estimate (\ref{eq: FactEst}) reduces to the elementary
estimate (\ref{eq: BVFactEst}). If $(B_{t})_{0\leqslant t\leqslant T}$
is a multidimensional Brownian motion, then its (pathwise) signature
coincides with the sequence of iterated stochastic integrals defined
in the sense of Stratonovich.
\end{example}
It is a fundamental result (cf. \cite{HL10} and \cite{BGLY16}) that every weakly geometric rough path over a real Banach space is uniquely determined by
its signature up to tree-like pieces. In addition, it is a consequence
of the weakly geometric property that any given component of signature
can be embedded into arbitrary higher degree components by raising tensor
powers (cf. \cite{CLN18}). Therefore, the \textit{tail} of signature (in the asymptotics
as degree tends to infinity) encodes essentially all information about
the underlying path.

In view of the factorial estimate (\ref{eq: FactEst}), a natural
quantity one can construct from the tail of signature is the normalized
component $((n/p)!\|X_{0,T}^{n}\|)^{p/n}$ as $n\rightarrow\infty.$
Since signature components can vanish infinitely often, we are led
to considering the functional 
\begin{equation}
L_{p}({\bf X})\triangleq\limsup_{n\rightarrow\infty}\left(\left(\frac{n}{p}\right)!\left\Vert X^{n}_{0,T}\right\Vert \right)^{\frac{p}{n}}.\label{eq: DefTail}
\end{equation}
\begin{comment}
As we explained in the introduction, we believe that $L_{p}({\bf X})$
is intimately related to certain notion of local $p$-variation 
of ${\bf X}$ defined in the same way as (\ref{eq: p-var}) but over arbitrarily fine scales of the path instead of taking supremum
over all finite partitions.
\end{comment}
\begin{comment}
 more precisely, 
\begin{equation}
\|{\bf X}\|_{\text{local-}p\text{-var}}\triangleq\lim_{\delta\rightarrow0}\sum_{k=1}^{\lfloor p\rfloor}\left(\sup_{\text{mesh}({\cal P})\leqslant\delta}\sum_{t_{i}\in{\cal P}}\|X_{t_{i-1},t_{i}}^{k}\|^{\frac{p}{k}}\right)^{\frac{k}{p}}.\label{eq: LocalPVar}
\end{equation}
\end{comment}

Our goal is to investigate at a quantitative level how the quantity $L_p(\bf X)$ is related to certain notion of local $p$-variation of $\mathbf{X}$
for the simplest type of rough paths known as pure rough paths. These are
straight forward analogues of line segments in the rough path context,
and they form the very first non-trivial class of rough paths for
the underlying problem.

\subsection{\label{subsec: Pure}Pure rough paths and formulation of main result}

Now we give the precise definition of the aforementioned class of rough paths that we will be working wtih. Let $m\geqslant1$ be a given integer.
\begin{defn}
A \textit{pure $m$-rough path} is a weakly geometric rough path
of the form 
\[
{\bf X}_{t}=\exp(tl)\in G^{(m)}(V),\ \ \ 0\leqslant t\leqslant1,
\]
where $l\in{\cal L}^{(m)}(V)$ is a Lie polynomial of degree $m$.
\end{defn}
\begin{example}
When $m=1,$ a pure $1$-rough path is simply a line segment in $V.$
\end{example}
We list a few basic properties of pure rough paths that are relevant
to us and leave the proofs in the appendix so as not to distract the
reader from the main picture.
\begin{prop}
\label{prop: LocalPVar}A pure $m$-rough path ${\bf X}_{t}=\exp(tl)$
is a rough path with roughness $m$ in the sense of Definition \ref{def: RoughPath}.
In addition, the local $m$-variation of ${\bf X}$ coincides with
the norm of the highest degree component of $l$, in the sense that
\[
\lim_{\delta\rightarrow0}\sum_{k=1}^{m}\left(\inf_{{\rm mesh}({\cal P})\leqslant\delta}\sum_{t_{i}\in{\cal P}}\|X_{t_{i-1},t_{i}}^{k}\|^{\frac{m}{k}}\right)^{\frac{k}{m}}=\|\pi_{m}(l)\|,
\]
where $\pi_{m}:T^{(m)}(V)\rightarrow V^{\otimes m}$ is the canonical
projection.
\end{prop}
\begin{rem}
We do not explicitly define local $p$-variation for general $p$-rough path because we are not aware of the most natural way of doing so. However, in the context of pure rough paths, whichever natural way of definition gives the same quantity $\|\pi_{m}(l)\|$
making it a canonical intrinsic property of the pure rough path. Indeed, the conclusion of Proposition \ref{prop: LocalPVar} remains unchanged if one replaces the "infimum" with a "supremum", or taking any a priori sequence of partitions whose mesh size tends to zero, or replacing the outer sum by taking maximum over degrees $1\leqslant k\leqslant m$. 
\end{rem}
\begin{prop}
\label{prop: SigPure}Let ${\bf X}_{t}=\exp(tl)$ be a pure $m$-rough
path. Then its signature is equal to $\exp(l)$ where
the exponential is now taken over the infinite tensor algebra $T((V))$.
In addition, up to tree-like equivalence this is the only weakly geometric
rough path whose signature is $\exp(l)$.
\end{prop}
In the case of pure rough paths, there are partial clues suggesting that the relationship between
the signature tail asymptotics and the local $m$-variation is as
simple and neat as stated in the following conjectural formula. This can be
viewed as an extension of the formula (\ref{eq: BVLengthConj}) in
the bounded variation case, and it is also consistent with what we see in
the Brownian motion case (cf. \cite{BG17}).
\begin{conjecture}
\label{conj: LengConjPure}For every pure $m$-rough path ${\bf X}_{t}=\exp(tl)\in G^{(m)}(V),$
the tail asymptotics quantity $L_{m}({\bf X})$ of signature equals the local $m$-variation
of ${\bf X}$. In view of Proposition \ref{prop: LocalPVar}, that
is $L_{m}({\bf X})=\|\pi_{m}(l)\|.$
\end{conjecture}
As a first major step towards understanding this problem, our main result can be summarized as a uniform upper and lower estimate
of $L_{m}({\bf X})$ in terms of $\|\pi_{m}(l)\|$ for pure $m$-rough
paths.
\begin{thm}
\label{thm:MainThm} Let $V$ be a finite dimensional Banach
space, and let every tensor product $V^{\otimes n}$ be equipped with the associated projective tensor norm. Then for each
$m\geqslant1$, there exists a constant $c(m,d)\in(0,1]$ depending
only on $m$ and $d\triangleq\dim V,$ such that 
\[
c(m,d)\|\pi_{m}(l)\|\leqslant L_{m}({\bf X})\leqslant\|\pi_{m}(l)\|
\]
for all pure $m$-rough paths ${\bf X}_{t}=\exp(tl)\in G^{(m)}(V)$.
The factor $c(m,d)$ admits an explicit lower estimate
\[
c(m,d)\geqslant \Lambda_{d}^{-m}\cdot2^{-(\nu_{m,d}!)^{\gamma\nu_{m,d}}},
\]
where $\Lambda_{d}$ is a constant depending only on $d$, $\nu_{m,d}\triangleq\dim{\cal L}_{m}(V)$,
and $\gamma>1$ is a universal constant.

In addition, if $V=\mathbb{R}^2$ is equipped with the $l^1$-norm with respect to the canonical basis, then for degrees $m=2,3,$ we further have $c(m,d)=1$, showing that
Conjecture \ref{conj: LengConjPure} holds for these cases. The same conclusion holds for some cases in degrees $m=4,5$ as well.
\end{thm}
\begin{rem}
When $m=1$, Conjecture \ref{conj: LengConjPure} boils down to the bounded variation formula (\ref{eq: BVLengthConj}) which holds trivially in this case since the underlying path is now a classical line segment.
\end{rem}

Although the main problem and results are motivated from rough path theory, we also give a parallel algebraic formulation which might raise potential interests in other fields.\\
\\
\textbf{Conjecture \ref{conj: LengConjPure}'.} \textit{Let $(V,\|\cdot\|)$ be a finite dimensional Banach space, and let the tensor products be equipped with some given reasonable tensor algebra norms. Then for any Lie polynomial $l$, the following asymptotics formula holds true:
\[
\limsup_{n\rightarrow\infty}\left(\left(\frac{n}{m}\right)!\|\pi_{n}(\exp(l))\|\right)^{\frac{m}{n}}=\|\pi_{m}(l)\|,
\]where $m$ is the degree of $l$.}\\
\\
\textbf{Theorem \ref{thm:MainThm}'.} \textit{Let $(V,\|\cdot\|)$ be a $d$ dimensional Banach space, and let the tensor products be equipped with the associated projective tensor norm. Then for each $m\geqslant1$, there exists a constant $c(m,d)\in(0,1]$ depending only on $m$ and $d$, such that for any Lie polynomial $l$ of degree $m$, the following estimate holds true:}
\[
c(m,d)\|\pi_{m}(l)\|\leqslant\limsup_{n\rightarrow\infty}\left(\left(\frac{n}{m}\right)!\|\pi_{n}(\exp(l))\|_{V^{\otimes n}}\right)^{\frac{m}{n}}\leqslant\|\pi_{m}(l)\|.
\]The factor $c(m,d)$ admits an explicit lower estimate and for some lower degree cases $c(m,d)=1$ giving the sharp result, precisely as stated in Theorem \ref{thm:MainThm}.

\section{\label{sec: SpecExam}Some special examples and heuristic calculations}

Before developing the proof of Theorem \ref{thm:MainThm}, we examine a few special examples
in order to get a better sense of the problem.

In the first place, the problem is trivial when (and only when) $\mathbf{X}_{t}$
is defined by a homogeneous polynomial. More precisely, if $\mathbf{X}_{t}=\exp(tl)$
with $l\in V^{\otimes m}$, it is immediate that 
\[
X^{n}=\pi_{n}(\exp(l))=\sum_{k=0}^{\infty}\frac{1}{k!}\pi_{n}(l^{\otimes k})=\begin{cases}
\frac{1}{(n/m)!}l^{\otimes(n/m)}, & m\mid n,\\
0, & m\nmid n.
\end{cases}
\]
Therefore, 
\begin{equation}\label{eq: HomCase}
L_{m}({\bf X})=\lim_{k\rightarrow\infty}\left(k!\|X^{km}\|\right)^{\frac{1}{k}}=\lim_{k\rightarrow\infty}\left(\|l^{\otimes k}\|^{\frac{1}{k}}\right)=\|l\|,
\end{equation}
and Conjecture \ref{conj: LengConjPure} holds
trivially  for ${\bf X}_{t}$.

A less trivial example is $l={\rm e}_{1}+[{\rm e}_{1},{\rm e}_{2}],$
in which we have
\begin{equation}
X^{2n}=\pi_{2n}(\exp(l))=\sum_{k=n}^{2n}\frac{1}{k!}\pi_{2n}\left(({\rm e}_{1}+[{\rm e}_{1},{\rm e}_{2}])^{\otimes k}\right).\label{eq: SimplestExam}
\end{equation}
A rather special observation in this example is that, the expansion of $\pi_{2n}(({\rm e}_{1}+[{\rm e}_{1},{\rm e}_{2}])^{\otimes k})$ is supported on disjoint sets of words for different $k$'s.
Suppose we work with the projective tensor norm induced from the standard
$l^{1}$-norm on $\mathbb{R}^{2}.$ 
%A rather special observation in
%this example is that, when we expand $\pi_{2n}(({\rm e}_{1}+[{\rm e}_{1},{\rm e}_{2}])^{\otimes k})$
%for all $n\leqslant k\leqslant2n,$ none of the words ${\rm e}_{i_{1}}\otimes\cdots\otimes{\rm e}_{i_{2n}}$
%will appear for more than once. 
It then follows that \[
\|X^{2n}\|=\sum_{k=n}^{2n}\frac{1}{k!}\left\Vert \pi_{2n}\left(({\rm e}_{1}+[{\rm e}_{1},{\rm e}_{2}])^{\otimes k}\right)\right\Vert \geqslant \frac{2^{n}}{n!}.
\]
In particular, 
\[
L_{2}({\bf X})\geqslant\limsup_{n\rightarrow\infty}\left(n!\|X^{2n}\|\right)^{\frac{1}{n}}=2=\|\pi_{2}(l)\|.
\]
Combining with the general upper bound to be established in Theorem
\ref{thm: upper} below, we see that Conjecture \ref{conj: LengConjPure}
 holds for ${\bf X}_{t}$.

However, it becomes much less clear how
similar calculations can be done even for the next simple candidate
$l={\rm e}_{1}+{\rm e}_{2}+[{\rm e}_{1},{\rm e}_{2}]$. Brute force
calculation does not give us much insight to proceed further. The main challenge of the problem lies in understanding the  complicated
interactions among different degree components of $l$ when looking
at the signature expansion at arbitrarily high degrees.

On the other hand, some extra mileage can still be achieved if we
work with the Hilbert-Schmidt tensor norm. Recall that the \textit{Hilbert-Schmidt tensor norm} over the tensor product of two Hilbert spaces $H_1,H_2$ is induced by \[
\langle v_{1}\otimes w_{1},v_{2}\otimes w_{2}\rangle_{H_{1}\otimes H_{2}}\triangleq\langle v_{1},v_{2}\rangle_{H_{1}}\cdot\langle w_{1},w_{2}\rangle_{H_{2}},\ \ \ v_1,v_2\in H_1,\ w_1,w_2\in H_2.
\]In this context, we can prove
the following result. We postpone the proof to Section \ref{sec: FreeLie}, whose strategy, based on orthogonality properties in free Lie algebras, is very different from the main approach of proving Theorem \ref{thm:MainThm}.
\begin{thm}
\label{thm: FreeLie}Let $V$ be a finite dimensional Hilbert space,
and let the tensor products be equipped with the induced Hilbert-Schmidt
tensor norm. Suppose that $\mathbf{X}_{t}=\exp(t(l_{a}+l_{b}))$,
where $l_{a}$, $l_{b}$ are homogeneous Lie polynomials of degrees
$a$, $b$ respectively for $a<b$. If $(b-a)/{\rm gcd}(a,b)$ is an
odd integer where "gcd" denotes the greatest common divisor, then Conjecture \ref{conj: LengConjPure} holds for ${\bf X}_{t}$.
\end{thm}
As an example, we immediately see that Conjecture \ref{conj: LengConjPure}
holds for $l={\rm e}_{1}+{\rm e}_{2}+[{\rm e}_{1},{\rm e}_{2}]$ under
the Hilbert-Schmidt tensor norm. However, the argument breaks down
if $(b-a)/{\rm gcd}(a,b)$ is an even number, or if $l$ has more than two homogeneous components.

The above special examples seem to suggest that, the key to getting
the lower bound is the concentration of the degree $km$ signature
expansion at the term $\pi_{m}(l)^{\otimes k}/k!$ as $k\rightarrow\infty.$
However, the picture can be much subtler in general. Some heuristic
estimates on magnitudes suggest that the signature expansion at degree $km$
is concentrated at a number of terms near $\pi_{m}(l)^{\otimes k}/k!$,
each possibly having comparable magnitudes. As $k\rightarrow\infty,$ the total
number of these terms seem to be of order $o(k)$, and there can be delicate
cancellations among them which are hard to analyze.

The main contribution of the present paper is to develop a general strategy
which on the one hand allows one to overcome the above difficulties
to some extent and on the other hand is specific enough to be implemented
computationally in order to generate explicit quantitative estimates in many
interesting examples.

\section{\label{sec: MainProof}Proof of Theorem \ref{thm:MainThm}}

Throughout the rest of this section, unless otherwise stated, let
$(V, \| \cdot \|)$ be a finite dimensional Banach space and let each tensor
product $V^{\otimes n}$ ($n\geqslant1$) be equipped with the projective
tensor norm. We work with a given pure $m$-rough path ${\bf X}_t = \exp(tl)$
defined by some Lie polynomial $l\in \mathcal{L}^{(m)}(V)$.  

We aim at studying
the relationship between the signature tail asymptotics of
${\bf X},$ defined by $L_{m}({\bf X})$ in (\ref{eq: DefTail}) with
$p=m,$ and the local $m$-variation of ${\bf X}$, which is also
equal to $\|\pi_{m}(l)\|$ by Proposition \ref{prop: LocalPVar}.
Our main result consists of uniform upper and lower estimates of $L_{m}({\bf X})$
in terms of $\|\pi_{m}(l)\|$. The techniques we develop for proving
the two estimates are drastically different. The upper estimate is
based on combinatorial arguments while the lower estimate relies
on the representation theory of complex semisimple Lie algebras.

\subsection{\label{subsec: Upper}The upper estimate}

We start by establishing the (sharp) upper bound. In this part,
more generality can be pursued: $V$ can be infinite dimensional,
tensor norms only need to be reasonable and $l$ need not be of Lie
type.
\begin{thm}
\label{thm: upper}We have the following upper estimate
\[
L_{m}({\bf X})\leqslant\|\pi_{m}(l)\|
\]
for all rough paths of the form ${\bf X}_t=\exp(tl)$ with $l$ being an arbitrary element in
$T^{(m)}(V).$
\end{thm}
Our proof of Theorem \ref{thm: upper} relies on a multivariate neoclassical
inequality proved by Friz-Riedel \cite{FR11}. The bivariate version
was proved by Hara-Hino \cite{HH10}.

\begin{lem}[cf. \cite{FR11}, Lemma 1]\label{lem: neo}Suppose that
$a_{1},\cdots,a_{m}>0$, $p\geqslant1$ and $n\in\mathbb{N}.$ Then
we have 
\[
\sum_{\substack{0\leqslant k_{1},\cdots,k_{m}\leqslant n\\
k_{1}+\cdots+k_{m}=n
}
}\frac{a_{1}^{k_{0}/p}\cdots a_{m}^{k_{m}/p}}{(k_{1}/p)!\cdots(k_{m}/p)!}\leqslant p^{m-1}\cdot\frac{(a_{1}+\cdots+a_{m})^{n/p}}{(n/p)!}.
\]

\end{lem}

We also need the following analytic lemma.
\begin{lem}
\label{lem: upper}Suppose that $0<\alpha<\beta\leqslant1$ and $a,b>0.$
Then we have 
\[
\limsup_{n\rightarrow\infty}\left((n\alpha)!\sum_{j=0}^{n}\frac{a^{j\alpha}b^{(n-j)\alpha}}{(j\beta)!\left((n-j)\alpha\right)!}\right)^{\frac{1}{n\alpha}}\leqslant b.
\]
\end{lem}
\begin{proof}
From Stirling's approximation, we know that 
\[
\frac{(j\alpha)!}{(j\beta)!}\sim\sqrt{\frac{\alpha}{\beta}}\left(\frac{\alpha^{\alpha}{\rm e}^{\beta-\alpha}}{\beta^{\beta}}\right)^{j}j^{(\alpha-\beta)j},\ \ \ j\rightarrow\infty
\]
In particular, given any arbitrary $\varepsilon>0,$ there exists
$J\geqslant1$ such that 
\[
\frac{(j\alpha)!}{(j\beta)!}\leqslant\varepsilon^{j}\ \ \ \forall j\geqslant J.
\]
It follows that
\begin{align}
\sum_{j=0}^{n}\frac{a^{j\alpha}b^{(n-j)\alpha}}{(j\beta)!\left((n-j)\alpha\right)!} & \leqslant\sum_{j=0}^{J-1}\frac{a^{j\alpha}b^{(n-j)\alpha}}{(j\beta)!\left((n-j)\alpha\right)!}+\sum_{j=J}^{n}\frac{\left(\varepsilon^{1/\alpha}a\right)^{j\alpha}b^{(n-j)\alpha}}{(j\alpha)!\left((n-j)\alpha\right)!}.\label{eq: UpperLem2}
\end{align}

To estimate the first term on the right hand side of (\ref{eq: UpperLem2}),
using Stirling's approximation again, it is easily seen that 
\[
\frac{(n\alpha)!}{(j\beta)!\left((n-j)\alpha\right)!}a^{j\alpha}b^{(n-j)\alpha}\leqslant Cn^{J\alpha}b^{n\alpha}\ \ \ {\rm for\ all}\ 0\leqslant j<J,
\]
where $C$ is a constant depending on $a,b,\alpha,\beta$ and $J$.
To estimate the second term on the right hand side of (\ref{eq: UpperLem2}),
using Lemma \ref{lem: neo} with $m=2$ and $p=1/\alpha$, we have
\[
\sum_{j=J}^{n}\frac{\left(\varepsilon^{1/\alpha}a\right)^{j\alpha}b^{(n-j)\alpha}}{(j\alpha)!\left((n-j)\alpha\right)!}\leqslant\frac{\left(\varepsilon^{1/\alpha}a+b\right)^{n\alpha}}{\alpha(n\alpha)!}.
\]
By substituting the above two estimates into (\ref{eq: UpperLem2}),
we have
\[
(n\alpha)!\sum_{j=0}^{n}\frac{a^{j\alpha}b^{(n-j)\alpha}}{(j\beta)!\left((n-j)\alpha\right)!}\leqslant Cn^{J\alpha}b^{n\alpha}+\frac{\left(\varepsilon^{1/\alpha}a+b\right)^{n\alpha}}{\alpha}.
\]

Therefore, by taking $n\rightarrow\infty,$ we arrive at
\[
\limsup_{n\rightarrow\infty}\left((n\alpha)!\sum_{j=0}^{n}\frac{a^{j\alpha}b^{(n-j)\alpha}}{(j\beta)!\left((n-j)\alpha\right)!}\right)^{\frac{1}{n\alpha}}\leqslant\varepsilon^{1/\alpha}a+b,
\]
which yields the result since $\varepsilon$ is arbitrary.
\end{proof}
With the help of the above two lemmas, we can now give the proof of
Theorem \ref{thm: upper}.

\begin{proof}[Proof of Theorem \ref{thm: upper}]

Given ${\bf X}_t=\exp(tl)$ with $l\in T^{(m)}(V)$, we write $l=l_{1}+\cdots+l_{m}$
where $l_{i}\in V^{\otimes i}.$ For each $n\geqslant1,$ the $n$-th
degree signature of ${\bf X}$ can be estimated by
\begin{align*}
\|X^{n}\| & =\|\pi_{n}\left(\exp(l)\right)\|\\
 & =\left\Vert \sum_{k=0}^{\infty}\frac{1}{k!}\pi_{n}\left((l_{1}+\cdots+l_{m})^{\otimes k}\right)\right\Vert \\
 & \leqslant\ \sum_{k=0}^{\infty}\frac{1}{k!}\sum_{\substack{1\leqslant i_{1},\cdots,i_{k}\leqslant m\\
i_{1}+\cdots+i_{k}=n
}
}\|l_{i_{1}}\|\cdots\|l_{i_{k}}\|\\
 & =\sum_{\substack{j_{1},\cdots,j_{m}\geqslant0\\
j_{1}+2j_{2}+\cdots+mj_{m}=n
}
}\frac{\|l_{1}\|^{j_{1}}\cdots\|l_{m}\|^{j_{m}}}{j_{1}!\cdots j_{m}!}.
\end{align*}
To reach the last equality, we have used a different way to count
terms that have a total degree of $n$ in the expansion of $(\|l_{1}\|+\cdots+\|l_{m}\|)^{k}.$
By applying change of variables $k_{r}=rj_{r}$ ($1\leqslant r\leqslant m$),
we arrive at 
\begin{equation}
\|X^{n}\|\leqslant\sum_{\substack{k_{1},\cdots,k_{m}\geqslant0\\
k_{1}+\cdots+k_{m}=n
}
}\frac{\|l_{1}\|^{k_{1}}\|l_{2}\|^{k_{2}/2}\cdots\|l_{m}\|^{k_{m}/m}}{k_{1}!(k_{2}/2)!\cdots(k_{m}/m)!}.\label{eq: upper}
\end{equation}

Next, for each fixed $k_{m},$ by using Lemma \ref{lem: neo} with
$p=m-1$, we see that 
\[
\sum_{\substack{k_{1},\cdots,k_{m-1}\geqslant0\\
k_{1}+\cdots+k_{m-1}=n-k_{m}
}
}\frac{\|l_{1}\|^{k_{1}}\cdots\|l_{m-1}\|^{k_{m-1}/(m-1)}}{k_{1}!\cdots\left(k_{m-1}/(m-1)\right)!}\leqslant(m-1)^{m-2}\cdot\frac{a^{(n-k_{m})/m}}{((n-k_{m})/(m-1))!},
\]
where 
\[
a\triangleq\left(\sum_{r=1}^{m-1}\|l_{r}\|^{\frac{m-1}{r}}\right)^{\frac{m}{m-1}}.
\]
By substituting this into (\ref{eq: upper}), we obtain
\[
\|X^{n}\|\leqslant(m-1)^{m-2}\sum_{k_{m}=0}^{n}\frac{a^{(n-k_{m})/m}\|l_{m}\|^{k_{m}/m}}{(k_{m}/m)!((n-k_{m})/(m-1))!}.
\]

Now the result follows from Lemma \ref{lem: upper} with $\alpha=1/m$,
$\beta=1/(m-1)$ and $b=\|l_{m}\|.$

\end{proof}
\begin{rem}
The upper estimate given by Theorem \ref{thm: upper} is
sharp, which can be easily seen by considering the case when $l$ is homogeneous
(i.e. when $l\in V^{\otimes m}$).
\end{rem}

\subsection{\label{subsec: LowerBdd}The core of the matter: Lie algebraic developments
and the lower estimate}

Now we turn our attention to establishing a matching lower bound, which is the core of the present paper. The philosophy of our main strategy can be briefly summarized as follows.

Our starting point is to look at the development of paths into a space
of automorphisms associated with a given representation of the tensor
algebra. This enables us to obtain an intermediate lower estimate of
$L_{m}({\bf X})$ in terms of eigenvalues of the highest degree Lie
component defining $\mathbf{X}$ under the given representation, and thus also allows us to eliminate the subtle contributions arising from the presence of lower degree Lie components.

The next key point, which
leads us to the main lower estimate, is to allow the representation
factor through a complex semisimple Lie algebra. In this way, the associated representation
theory enables us to study eigenvalues of the highest degree Lie polynomial at an explicit and quantitative level.
This is largely due to the presence of an abelian subalgebra (a so-called Cartan subalgebra) consisting
of semisimple elements, a basic feature of semisimple Lie algebras
that is quite different from nilpotent (or more generally, solvable)
Lie algebras. A crucial step towards making good use of such feature is to develop highest degree Lie polynomials into this Cartan subalgebra. 

Our plan of proving the main lower estimate is organized in the following way, which also underlines the main ingredients of our strategy.\\ 
\\
\textbf{Organization of this subsection.} In Section \ref{sec: LieAlgDev}, we introduce the notion of Lie algebraic developments, which is a main tool we will be using for proving our lower estimate. In Section \ref{sec: IntLower}, we prove an intermediate lower estimate using path developments and finite dimensional perturbation theory. Section \ref{sec: MainLower} is devoted to reaching the ultimate lower estimate from the intermediate one, and for this purpose it is further divided into four parts. Part I contains a quick review on several notions from the representation theory of complex semisimple Lie algebras that are needed in our approach. In Part II, we develop ways of mapping a space of homogeneous Lie polynomials into a Cartan subalgebra, by using basic root patterns from semisimple Lie theory. Part III is devoted to the proof of a  consistency lemma for certain polynomial systems, which is a crucial ingredient in order to obtain a uniform lower estimate. In Part IV, having all necessary ingredients at hand, we give the proof of our main lower estimate by designing appropriate Lie algebraic developments. In Section \ref{subsec: LowDeg}, we perform explicit calculations in low degree cases to demonstrate how our strategy can be implemented specifically, leading to the sharp lower bound in certain situations.

\subsubsection{Lie algebraic developments of rough paths}\label{sec: LieAlgDev}

To describe the necessary structures efficiently, we start with the
following definition.

\begin{defn}
\label{def: AlgDev}Let $V$ be a real or complex Banach space. A
\textit{Lie algebraic development} $\Phi$ of $V$ consists of a linear
map $F:V\rightarrow\mathfrak{g}$ into a complex Lie algebra $\mathfrak{g}$
and a representation $\rho:\mathfrak{g}\rightarrow{\rm End}(W)$ of
$\mathfrak{g}$ on a complex Banach space $W$ such that $\Phi=\rho\circ F$
is continuous, where ${\rm End}(W)$ denotes the space of continuous
linear transformations over $W$. The development $\Phi$ is said to be \textit{finite
dimensional} if $\mathfrak{g}$ and $W$ are both finite dimensional. In situations when the intermediate Lie algebra $\mathfrak{g}$ is not relevant, we simply refer to $\Phi:V\rightarrow\mathrm{End}(W)$ as a \textit{development}.
\end{defn}

\begin{rem}
When $V$ is real, linearity is understood over $\mathbb{R}$ by regarding a complex vector space as a real vector space in the obvious way. 
\end{rem}

Let $\Phi:V\rightarrow{\rm End}(W)$ be a given development. According to the universal property of the
projective tensor product (cf. \cite{LQ02}, Theorem 5.6.3), for
each $n\geqslant1,$ $\Phi$ induces a continuous linear map $\Phi^{(n)}:V^{\otimes n}\rightarrow{\rm End}(W)$
such that
\[
\Phi(v_{1}\otimes\cdots\otimes v_{n})=\Phi(v_{1})\cdots\Phi(v_{n})
\]
and 
\begin{equation}
\|\Phi^{(n)}\|_{V^{\otimes{n}}\rightarrow\mathrm{End(W)}}\leqslant\|\Phi\|_{V\rightarrow\mathrm{End}(W)}^{n}.\label{eq: UniProj}
\end{equation}
It follows that $\Phi$ induces a natural algebra homomorphism
from a subspace of $T((V))$ to ${\rm End}(W),$ which is defined by (still denoted
as $\Phi$)
\[
\Phi\left((\xi_{0},\xi_{1},\xi_{2},\cdots)\right)\triangleq\xi_{0}\cdot{\rm Id}+\sum_{n=1}^{\infty}\Phi^{(n)}(\xi_{n}),
\]
provided that the sum on the right hand side is convergent under the
operator norm on ${\rm End}(W).$ In addition, $\Phi$ descends to
a natural Lie algebra homomorphism from the free Lie algebra ${\cal L}(V)=\oplus_{n=1}^\infty \mathcal{L}_n(V)$ over $V$
 into ${\rm End}(W).$

Under the given development $\Phi$, every rough path $({\bf X}_{t})_{0\leqslant t\leqslant T}$
over $V$ can be developed onto ${\rm Aut}(W),$ the space of automorphisms
over $W,$ by solving the linear ODE 
\begin{equation}
\begin{cases}
d\Gamma_{t}=\Gamma_{t}\cdot\Phi(d{\bf X}_{t}),\ 0\leqslant t\leqslant T,\\
\Gamma_{0}={\rm Id}.
\end{cases}\label{eq: EqnDev}
\end{equation}
Using Picard's iteration, it is easily seen that 
\begin{align}
\Gamma_{t} & =\sum_{n=0}^{\infty}\int_{0<t_{1}<\cdots<t_{n}<t}\Phi(d{\bf X}_{t_{1}})\cdots\Phi(d{\bf X}_{t_{n}})\nonumber \\
 & =\sum_{n=0}^{\infty}\Phi^{(n)}\left(\int_{0<t_{1}<\cdots<t_{n}<t}d{\bf X}_{t_{1}}\otimes\cdots\otimes d{\bf X}_{t_{n}}\right)\nonumber \\
 & =\Phi(\mathbb{X}_{0,t}),\label{eq: FormDev}
\end{align}
where $\mathbb{X}_{0,t}$ is the Lyons' extension of $\mathbf{X}$ given by Theorem \ref{thm: LyonsExt}.
Note that by factorial decay inequality in the same theorem,
$\Phi(\mathbb{X}_{0,t})$ is well defined. In particular, we have $\Gamma_{T}=\Phi(S(\mathbf{X}))$.
\begin{rem}
In the above discussion, the intermediate Lie algebra $\mathfrak{g}$
and the complex structure appearing in Definition \ref{def: AlgDev} are not so relevant yet, and the structure used
here is simply a representation of the tensor algebra. Their
roles will become clear later on when we look for explicit quantitative
lower estimates for the signature.
\end{rem}
The viewpoint of developing Euclidean paths onto a Lie group was essentially
due to Cartan and had been used by many authors for geometric reasons.
We give an example which is related to studies on path signatures.
\begin{example}
Hambly-Lyons \cite{HL10} proved the asymptotics formula
(\ref{eq: BVLengthConj}) for $C^{1}$-paths by developing the underlying
path onto the space of constant negative curvature. Under the notion
of Lie algebraic developments, in their case $V=\mathbb{R}^{d},$ 
\[
\mathfrak{g}=\left\{ \left(\begin{array}{cc}
A & x\\
x^{T} & 0
\end{array}\right)\in{\rm Mat}(d+1;\mathbb{R}):A\in\mathfrak{so}(d),x\in\mathbb{R}^{d}\right\} 
\]
is the Lie algebra of the isometry group for the standard $d$-dimensional
hyperboloid. The embedding $F:V\rightarrow\mathfrak{g}$ is given by
\[
F(x)\triangleq\left(\begin{array}{cc}
0 & x\\
x^{T} & 0
\end{array}\right),\ \ \ x\in\mathbb{R}^{d},
\]
and $\rho:\mathfrak{g}\rightarrow{\rm End}(W)$ is the canonical matrix
representation with $W=\mathbb{R}^{d+1}.$ Rather than looking at
the developed path $\Gamma_{t}$ in the isometry group, the authors
worked with the trajectory on the hyperboloid traced out by the
action of $\Gamma_{t}$ on a base point of the hyperboloid. Their main philosophy, which
is rather geometric, is to make use of exotic properties of hyperbolic
geodesics which do not have Euclidean counterparts. Related results
by Lyons-Xu \cite{LX15} for studying signature inversion
and by Boedihardjo-Geng \cite{BG17} for studying tail asymptotics
of the Brownian signature, are also based on similar geometric insights.
In this hyperbolic picture, there is no need to work with complex
structure appearing in Definition \ref{def: AlgDev}.
\end{example}
In contrast to the hyperbolic geometric ideas, our approach deviates
from the aforementioned works by not projecting the path onto a base
manifold which the group acts on. Instead of following geometric intuitions,
we look at path developments from an algebraic viewpoint, which provides a more suitable framework for the implementation of representation-theoretic techniques. 

\subsubsection{An intermediate lower estimate}\label{sec: IntLower}

Using the notion of developments, we are led to a general lower estimate
which holds for arbitrary rough paths. A similar estimate already
appeared in \cite{BG17} for the hyperbolic development of Brownian
motion. Given any $p$-rough path $\mathbf{X}_t$ and $\lambda>0$, we use $\delta_\lambda(\mathbf{X}_t)$ to denote the dilated path $(1,\lambda X_t^1,\cdots,\lambda^{\lfloor p\rfloor}X_t^{\lfloor p\rfloor})$.
\begin{prop}
\label{prop: FirstLBdd}Let $({\bf X}_{t})_{0\leqslant t\leqslant T}$
be a $p$-rough path over some Banach space $V$. For any given
 non-zero development $\Phi:V\rightarrow{\rm End}(W),$ we have the
following lower estimate for the signature tail asymptotics of ${\bf X}$:
\begin{equation}
L_{p}({\bf X})\geqslant\limsup_{\lambda\rightarrow\infty}\frac{\log \|\Gamma_{T}^{\lambda}\|_{W\rightarrow W}}{\left(\lambda\|\Phi\|_{V\rightarrow{\rm End}(W)}\right)^{p}},\label{eq: FirstLBdd}
\end{equation}
where for $\lambda>0,$ $(\Gamma_{t}^{\lambda})_{0\leqslant t\leqslant T}$
denotes the development of the dilated path $\delta_\lambda(\mathbf{X}_t)$ under $\Phi,$ defined by the ODE (\ref{eq: EqnDev}).
\end{prop}
\begin{proof}
According to the formula (\ref{eq: FormDev}) for path developments,
we have 
\[
\Gamma_{T}^{\lambda}=\sum_{n=0}^{\infty}\lambda^{n}\Phi^{(n)}\left(X^{n}\right),
\]where $X^n$ is the degree $n$ component of the signature of $\mathbf{X}$.
For given $N\geqslant1,$ define 
\[
L_{N}\triangleq\sup_{n\geqslant N}\left(\left(\frac{n}{p}\right)!\|X^{n}\|\right)^{\frac{p}{n}},
\]which is finite according to the factorial estimate (\ref{eq: FactEst}).
Note that if $L_{N}=0$ for some $N$, then the right hand side of (\ref{eq: FirstLBdd}) is zero since $\Gamma_T^\lambda$ becomes polynomial in $\lambda$ in this case. Therefore, we will assume that $L_N>0$ for all $N$. 
It then follows from (\ref{eq: UniProj}) that 
\begin{align}
 & \|\Gamma_{T}^{\lambda}\|_{W\rightarrow W}\nonumber \\
 & \leqslant\sum_{n=0}^{N-1}(\lambda\|\Phi\|)^{n}\|X^{n}\|+\sum_{n=N}^{\infty}(\lambda\|\Phi\|)^{n}\|X^{n}\|\nonumber \\
 & \leqslant\sum_{n=0}^{N-1}(\lambda\|\Phi\|)^{n}\|X^{n}\|+\sum_{n=N}^{\infty}\frac{\left(\lambda^{p}\|\Phi\|^{p}L_{N}\right)^{n/p}}{(n/p)!}\nonumber \\
 & =\sum_{n=0}^{\infty}\frac{\left(\lambda^{p}\|\Phi\|^{p}L_{N}\right)^{n/p}}{(n/p)!}+\sum_{n=0}^{N-1}\left((\lambda\|\Phi\|)^{n}\|X^{n}\|-\frac{\left(\lambda^{p}\|\Phi\|^{p}L_{N}\right)^{n/p}}{(n/p)!}\right),\label{eq: EstGam}
\end{align}where for notational simplicity we have omitted the subscript for the operator norm of $\Phi$.

To understand the asymptotic behaviour of the right hand side as $\lambda\rightarrow\infty,$
we first consider the explicit function defined by
\[
f(x)\triangleq\sum_{n=0}^{\infty}\frac{x^{n/p}}{(n/p)!}.
\]
We claim that 
\begin{equation}
f(x)\leqslant(p+1)x{\rm e}^{x}\ \ \ {\rm for}\ {\rm all}\ x\geqslant1.\label{eq: SumGrowth}
\end{equation}
Indeed, for each $m\geqslant0,$ define $R_{m}$ to be the set of
non-negative real numbers $r<p$ such that $mp+r$ is an integer.
Then $R_{m}\subseteq[0,p)$ consists of no more than $p+1$ elements.
Therefore, 
\begin{align*}
f(x) & =\sum_{m=0}^{\infty}\sum_{r\in R_{m}}\frac{x^{m+r/p}}{(m+r/p)!}\leqslant\sum_{m=0}^{\infty}\frac{x^{m}}{m!}\sum_{r\in R_{m}}x^{r/p}\leqslant(p+1)x\sum_{m=0}^{\infty}\frac{x^{m}}{m!}=(p+1)x{\rm e}^{x}.
\end{align*}

By applying (\ref{eq: SumGrowth}) to the first term on the right
hand side of (\ref{eq: EstGam}) and denoting the second term as $q_{N}(\lambda),$
we obtain that 
\[
\|\Gamma_{T}^{\lambda}\|_{W\rightarrow W}\leqslant(p+1)\lambda^{p}\|\Phi\|^{p}L_{N}\exp\left(\lambda^{p}\|\Phi\|^{p}L_{N}\right)+q_{N}(\lambda).
\]
Note that $q_{N}(\lambda)$ has polynomial growth in $\lambda.$ Therefore,
by taking $\lambda\rightarrow\infty,$ we have 
\[
\limsup_{\lambda\rightarrow\infty}\frac{\log\|\Gamma_{T}^{\lambda}\|_{W\rightarrow W}}{\left(\lambda\|\Phi\|\right)^{p}}\leqslant L_{N}.
\]
Since $N$ is arbitrary, we conclude that 
\[
\limsup_{\lambda\rightarrow\infty}\frac{\log\|\Gamma_{T}^{\lambda}\|_{W\rightarrow W}}{\left(\lambda\|\Phi\|\right)^{p}}\leqslant\inf_{N\geqslant1}L_{N}=L_{p}({\bf X}).
\]
\end{proof}
At first glance, the estimate (\ref{eq: FirstLBdd}) does not seem
to be as useful as it will be. We now unwind the shape of its right
hand side in the context of pure rough paths. From now on, we confine
ourselves in finite dimensional developments, which is the main situation
where useful calculations can be done explicitly.
\begin{thm}
\label{thm: IntLBdd}Let ${\bf X}_{t}=\exp(tl)$ be a pure $m$-rough
path with $l\in{\cal L}^{(m)}(V).$ For any given finite dimensional development $\Phi:V\rightarrow{\rm End}(W),$ we have 
\begin{equation}
L_{m}({\bf X})\geqslant\frac{\sup\left\{ {\rm Re}(\mu):\mu\in\sigma(\Phi(l_{m}))\right\} }{\|\Phi\|_{V\rightarrow\mathrm{End}(W)}^{m}},\label{eq: IntLBdd}
\end{equation}
where $l_{m}\triangleq\pi_{m}(l)$ is the highest degree component
of $l,$ and $\sigma(\Phi(l_{m}))$ denotes the set of eigenvalues
of $\Phi(l_{m})\in{\rm End}(W)$.
\end{thm}
\begin{proof}
The proof is an application of perturbation theory in finite dimensions.
Let $\mu$ be an eigenvalue of $\Phi(l_{m})$ and write $l=l_{1}+\cdots+l_{m}$
as the sum of homogeneous components. According to \cite{Kato},
Chapter 2, Theorem 5.1 and Theorem 5.2 applied to the continuous family
\[
(0,\infty)\ni\lambda\mapsto T(\lambda)\triangleq\Phi(l_{m})+\frac{1}{\lambda}\Phi(l_{m-1})+\cdots+\frac{1}{\lambda^{m-1}}\Phi(l_{1})\in{\rm End}(W)
\]
of bounded linear transformations, we know that there exists a complex
valued continuous function $\phi(\lambda)$, such that $\phi(\lambda)$
is an eigenvalue of $T(\lambda)$ for all $\lambda$ and $\phi(\lambda)\rightarrow\mu$
as $\lambda\rightarrow\infty.$ On the other hand, let $(\Gamma_{t}^{\lambda})_{0\leqslant t\leqslant1}$
be the development of the dilated path $\delta_{\lambda}({\bf X}_{t})$
under $\Phi.$ By (\ref{eq: FormDev}) and the definition of operator
norm, we have 
\begin{align*}
\|\Gamma_{1}^{\lambda}\|_{W\rightarrow W} & =\|\Phi\left(\delta_{\lambda}(\exp(l))\right)\|_{W\rightarrow W}=\|\exp(\Phi(\delta_{\lambda}(l)))\|_{W\rightarrow W}\\
 &=\|\exp(\lambda^{m}T(\lambda))\|_{W\rightarrow W} \geqslant\left|\exp\left(\lambda^{m}\phi(\lambda)\right)\right|=\exp\left(\lambda^{m}{\rm Re}(\phi(\lambda))\right).
\end{align*}
Therefore, 
\[
\frac{\log\|\Gamma_{1}^{\lambda}\|_{W\rightarrow W}}{\lambda^{m}}\geqslant{\rm Re}(\phi(\lambda))
\]
for all $\lambda>0.$ Now the result follows from Proposition \ref{prop: FirstLBdd}
by taking $\lambda\rightarrow\infty.$
\end{proof}
\begin{rem}
Note that the right hand side of (\ref{eq: IntLBdd}) does not depend
on lower degree components of $l$. In other words, Theorem \ref{thm: IntLBdd}
provides a possible way of eliminating the complicated interactions
among different degree components of $l$ in the signature tail asymptotics. Nonetheless, 
it is \textit{not} true that this fact allows us to conclude Conjecture \ref{conj: LengConjPure} 
directly from the homogeneous case (i.e. when $l=l_m$) for which we know the result holds trivially (cf. (\ref{eq: HomCase}) in Section \ref{sec: SpecExam}). The subtle point is that, as suggested by (\ref{eq: IntLBdd}), the elimination of lower degree effects is only achieved through a given development $\Phi$. Therefore, even though we know the result holds for the homogeneous case, one needs to see that the lower bound can be attained at some specific choice of developments in order to conclude the result for the inhomogeneous case. Designing such developments is the main goal of what follows.

\end{rem}

\subsubsection{The main lower estimate}\label{sec: MainLower}

In view of Theorem \ref{thm: IntLBdd}, to obtain useful lower
estimates on $L_{m}({\bf X}),$ one needs to design suitable Lie algebraic
developments under which we can estimate eigenvalues of $l_{m}$ effectively.
This is where the intermediate Lie algebra $\mathfrak{g}$ and the
complex structure in Definition \ref{def: AlgDev} come into play.
In particular, we will choose $\mathfrak{g}$ to be a finite dimensional
complex semisimple Lie algebra and rely on the associated representation
theory.

\paragraph*{I. Notions from the representation theory of complex semisimple
Lie algebras}
\label{sec: LowerI}
\addcontentsline{toc}{paragraph}{\nameref{sec: LowerI}}

$\ $\\
\\
To explain how the semisimple structure plays a role, it is helpful
to first recall some relevant notions from Lie theory. We refer the
reader to \cite{Humphreys72} for more details. Unless otherwise
stated, all Lie algebras and representations are finite dimensional
over the complex field. The main benefit of this setting is the existence
of eigenvalues for linear transformations, which significantly simplifies the
associated representation theory.
\begin{defn}
A complex Lie algebra $\mathfrak{g}$ is called \textit{semisimple}
if it is isomorphic to a direct sum $\mathfrak{g}\cong\mathfrak{g}_{1}\oplus\cdots\oplus\mathfrak{g}_{r}$ of Lie algebras,
where each summand $\mathfrak{g}_{i}$ is \textit{simple} in the sense
that it does not contain non-trivial proper ideals.
\end{defn}
It can be shown that semisimpleness is equivalent to the non-degeneracy
of the \textit{Killing form}, which is the bilinear form $B:\mathfrak{g}\times\mathfrak{g}\rightarrow\mathbb{C}$
defined by 
\[
B(X,Y)\triangleq{\rm Tr}\left({\rm ad}(X)\circ{\rm ad}(Y)\right),
\]
where ${\rm Tr}$ means taking trace and ${\rm ad}:\mathfrak{g}\rightarrow{\rm End}(\mathfrak{g})$
denotes the \textit{adjoint representation} of $\mathfrak{g}$ given
by ${\rm ad}(X)(Z)\triangleq[X,Z].$

A central concept in semisimple Lie theory that is also crucial for
us is the following.
\begin{defn}
A \textit{Cartan subalgebra} of $\mathfrak{g}$ is a subspace $\mathfrak{h}\subseteq\mathfrak{g}$
such that:\\
\\
(i) $\mathfrak{h}$ is a maximal abelian subalgebra of $\mathfrak{g}$;\\
(ii) for each $H\in\mathfrak{h},$ the linear transformation ${\rm ad}(H)\in{\rm End}(\mathfrak{g})$
is semisimple (over $\mathbb{C}$ this is equivalent to being diagonalizable).
\end{defn}
For a complex semisimple Lie algebra, a Cartan subalgebra always exists
and it is unique up to conjugation in $\mathfrak{g}$. Let $\mathfrak{h}$ be a given Cartan subalgebra
of $\mathfrak{g}$. By its definition and a standard application of
linear algebra, given an arbitrary representation $\rho:\mathfrak{g}\rightarrow\mathrm{End}(W)$,
all elements of $\mathfrak{h}$ are simultaneously diagonalizable
when viewed as linear transformations over $W$ under $\rho.$ More
specifically, a complex linear functional $\mu\in\mathfrak{h}^{*}$
is called a \textit{weight} for the given representation $\rho$ if
the subspace
\begin{equation}\label{eq: WeightSpace}
W^{\mu}\triangleq\left\{ w\in W:\rho(H)w=\mu(H)w\ \mathrm{for\ all}\  H\in\mathfrak{h}\right\} 
\end{equation}
is non-trivial. It follows that there are at most finitely many weights for $\rho$.
 Denote their collection by $\Pi(\rho)$. The space
$W$ then admits the decomposition (simultaneous diagonalization) $W=\oplus_{\mu\in\Pi(\rho)}W^{\mu},$
in which for every $H\in\mathfrak{h}$, $W^{\mu}$ 
is an eigenspace of $\rho(H)$ with eigenvalue $\mu(H)$ ($\mu\in\Pi(\rho)$).

Indeed, much more can be said in the semisimple setting. We first
look at the adjoint representation of $\mathfrak{g}$. Given a complex
linear functional $\alpha\in\mathfrak{h}^{*},$ define the subspace
\[
\mathfrak{g}^{\alpha}\triangleq\left\{ X\in\mathfrak{g}:{\rm ad}(H)(X)=\alpha(H)X\ {\rm for}\ {\rm all}\ H\in\mathfrak{h}\right\}
\]in the same way as (\ref{eq: WeightSpace}).
It is easy to verify that $\mathfrak{g}^{0}=\mathfrak{h},$ and $[\mathfrak{g}^{\alpha},\mathfrak{g}^{\beta}]\subseteq\mathfrak{g}^{\alpha+\beta}$
for all $\alpha,\beta\in\mathfrak{h}^{*}.$ A complex linear functional
$\alpha\in\mathfrak{h}^{*}$ is called a \textit{root} of $\mathfrak{g}$
with respect to $\mathfrak{h}$ if it is a weight for the adjoint
representation, i.e. if $\mathfrak{g}^{\alpha}\neq\{0\}$. In this
case, $\mathfrak{g}^{\alpha}$ is called the\textit{ root space} associated
with the root $\alpha.$ As before, there are at most finitely many
roots. A basic result in semisimple Lie theory is the following so-called
\textit{root space decomposition}.
\begin{thm}
Let $\Delta\subseteq\mathfrak{h}^*$ be the set of nonzero roots with respect to a given
Cartan subalgebra $\mathfrak{h}$. Then $\mathfrak{g}$ can be written
as the direct sum 
\[
\mathfrak{g}=\mathfrak{h}+\sum_{\alpha\in\Delta}\mathfrak{g}^{\alpha}.
\]
In addition, $\dim\mathfrak{g}^{\alpha}=1$ for each $\alpha\in\Delta,$
and if $\alpha,\beta$ are two roots with $\alpha+\beta\in\Delta,$
then 
\begin{equation}
[\mathfrak{g}^{\alpha},\mathfrak{g}^{\beta}]=\left\{\begin{array}{ll}\mathfrak{g}^{\alpha+\beta}, \quad \text{if} \quad \alpha+\beta\in \Delta,\\
\{0\}, \quad \text{otherwise}.
\end{array}
\right.\label{eq: GradRoot}
\end{equation}
\end{thm}
It is possible to study general representations using the structure
of roots. Before stating relevant results, we need a few more definitions. Let $E$ be
the vector space generated by $\Delta$ over $\mathbb{R}.$ A subset
$\Delta_{0}$ of $\Delta$ is called a \textit{base} if:\\
\\
(i) $\Delta_{0}$ is a basis of $E$;\\
(ii) each root $\beta\in\Delta$ can be expressed as $\beta=\sum_{\alpha\in\Delta_{0}}k_{\alpha}\alpha$
with integral coefficients $k_{\alpha}$ either being all non-negative
or all non-positive.\\
\\
The roots in $\Delta_{0}$ are called \textit{simple roots}. The choice
of $\Delta_{0}$ is not unique but its cardinality is. The Lie algebra
$\mathfrak{g}$ is said to have \textit{rank} $m$ if $\Delta_{0}$
has $m$ elements, which is equivalent to saying that $\dim_{\mathbb{C}}\mathfrak{h}=m.$
Let $\Delta_{0}=\{\alpha_{1},\cdots,\alpha_{m}\}$ be a given set
of simple roots. The Killing form $B$ restricted to $\mathfrak{h}$
is also non-degenerate. It follows that, for each $\alpha_{i}\in\Delta_{0},$
there exists $T_{i}\in\mathfrak{h}$ such that $\alpha_{i}(H)=B(T_{i},H)$
for all $H\in\mathfrak{h}.$ We define the normalized element $H_{i}\triangleq2T_{i}/B(T_{i},T_{i})$.
\begin{defn}
A linear functional $\lambda\in\mathfrak{h}^{*}$ is called a \textit{dominant
integral functional} if all $\lambda(H_{i})$ ($1\leqslant i\leqslant m$)
are non-negative integers. The set $\{\lambda_{1},\cdots,\lambda_{m}\}$
of \textit{fundamental dominant integral functionals} are defined
by the duality relation $\lambda_{i}(H_{j})=\delta_{ij}.$
\end{defn}
The main result in the representation
theory of complex semisimple Lie algebras is stated as follows. Recall that a representation
$\rho:\mathfrak{g}\rightarrow{\rm End}(W)$ is \textit{irreducible}
if $W$ does not contain non-trivial proper $\mathfrak{g}$-invariant
subspaces.
\begin{thm}
\label{thm: ClasRep}Let $\mathfrak{g}$ be a complex semisimple Lie
algebra with a given Cartan subalgebra $\mathfrak{h}$ and a given
base $\Delta_{0}$ of simple roots. There is a one-to-one correspondence
between dominant integral functionals and isomorphism classes of (finite
dimensional) irreducible representations.
\end{thm}
We must point out that representation theory provides richer quantitative
information than the statement of the above classification theorem itself.
A main consequence of the theory which is relevant to us, is that
for each dominant integral functional $\lambda,$ the set of weights
for the associated irreducible representation can be described in
a rather quantitative way, making the computation of eigenvalues of
elements in $\mathfrak{h}$ quite tractable. We use an important example
to illustrate this point, in which all the aforementioned notions
and results can also be worked out explicitly. The implementation
of our main technique is largely based on this example.

Consider $\mathfrak{g}=\mathfrak{sl}(m,\mathbb{C})$ ($m\geqslant2$),
the set of $m\times m$ complex matrices with zero trace. Then $\mathfrak{g}$
is a complex semisimple (in fact, simple) Lie algebra of rank $m-1$.
A Cartan subalgebra $\mathfrak{h}$ can be chosen as the subspace
of diagonal matrices in $\mathfrak{g}$. For each $1\leqslant i\leqslant m,$
define $\mu_{i}\in\mathfrak{h}^{*}$ to be linear functional of taking
the $i$-th diagonal entry. Then the set of nonzero roots with respect
to $\mathfrak{h}$ is given by 
\[
\Delta=\left\{ \alpha_{i,j}\triangleq\mu_{i}-\mu_{j}:1\leqslant i\neq j\leqslant m\right\} .
\]
In addition, for each $i\neq j,$ the root space $\mathfrak{g}^{\alpha_{i,j}}=\mathbb{C}\cdot E_{i,j},$
where $E_{i,j}$ is the matrix whose $(i,j)$-entry is $1$ and all
other entries are $0$'s. To summarize, the root space decomposition
takes the form 
\[
\mathfrak{g}=\mathfrak{h}+\sum_{1\leqslant i\neq j\leqslant m}\mathbb{C}\cdot E_{i,j}.
\]

A base of simple roots can be chosen as
\begin{equation}
\Delta_{0}=\left\{ \alpha_{i}\triangleq\mu_{i}-\mu_{i+1}:1\leqslant i\leqslant m-1\right\} .\label{eq: BaseSL}
\end{equation}
For each simple root $\alpha_{i}\in\Delta_{0},$ the associated $H_{i}\in\mathfrak{h}$
is given by the diagonal matrix in which the $i$-th diagonal entry
is $1$, the $(i+1)$-th diagonal entry is $-1,$ and all other entries
are zero. To describe the corresponding classification theorem for
irreducible representations of $\mathfrak{sl}(m,\mathbb{C})$, we
first recall that, a given linear transformation $T$ over a vector
space $W$ induces for each $k\geqslant1,$ a natural linear transformation
$T^{\otimes k}$ on the tensor product $W^{\otimes k}$ satisfying
\[
T^{\otimes k}(w_{1}\otimes\cdots\otimes w_{k})=\sum_{i=1}^{k}w_{1}\otimes\cdots\otimes T(w_{i})\otimes\cdots\otimes w_{k},
\]
which also descends to a natural linear transformation $T^{\wedge k}$
on the $k$-th exterior power $\Lambda^k(W)$ of $W$. It follows that, a given
representation $\rho:\mathfrak{g}\rightarrow{\rm End}(W)$ of a Lie
algebra $\mathfrak{g}$ induces for each $k\geqslant1,$ a representation
$\rho^{\otimes k}$ (respectively, $\rho^{\wedge k}$) on $W^{\otimes k}$
(respectively, on $\Lambda^k(W)$) in the natural way. The representation
theory of $\mathfrak{g}=\mathfrak{sl}(m,\mathbb{C})$ can be summarized
as the following theorem. Note that $\mathfrak{g}$ acts on $W\triangleq\mathbb{C}^{m}$
in the canonical way by matrix multiplication. We call this canonical
matrix representation $\rho.$ For $1\leqslant k\leqslant m-1,$ denote
$W_{k}\triangleq \Lambda^k(W).$
\begin{thm}
\label{thm: RepSL}Let a Cartan subalgebra and a base of simple roots
for $\mathfrak{g}=\mathfrak{sl}(m,\mathbb{C})$ be given as above.\\
\\
(1) The set $\{\lambda_{1},\cdots,\lambda_{m-1}\}$ of fundamental
dominant integral functionals are given by $\lambda_{k}=\mu_{1}+\cdots+\mu_{k}$
($1\leqslant k\leqslant m-1$). For each $k$, the irreducible representation
associated with $\lambda_{k}$ is given by $\rho^{\wedge k}:\mathfrak{g}\rightarrow{\rm End}(W_{k}),$
whose set of weights is precisely
\[
\Pi(\lambda_{k})=\left\{ \mu_{i_{1}}+\cdots+\mu_{i_{k}}:1\leqslant i_{1}<\cdots<i_{k}\leqslant m\right\} .
\]
(2) For each dominant integral functional $\lambda=a_{1}\lambda_{1}+\cdots+a_{m-1}\lambda_{m-1}$
with $a_{i}$'s being non-negative integers, the representation $\rho^{\lambda}:\mathfrak{g}\rightarrow{\rm End}(W_{1}^{\otimes a_{1}}\otimes\cdots\otimes W_{m-1}^{\otimes a_{m-1}})$
contains exactly one copy of the irreducible representation associated
with $\lambda$, whose set of weights is a subset of 
\[
\left\{ \sum_{k=1}^{m-1}\sum_{j=1}^{a_{k}}\nu_{k,j}:\nu_{k,j}\in\Pi(\lambda_{k})\right\} .
\]
\end{thm}
\begin{rem}
In the second part of the above theorem, by using Schur polynomials and
Young tableaux, it is possible to identify  the precise copy of irreducible
representation contained in the tensor product representation as well
as the associated set of weights. However, at the moment we do not see
the need of pursuing this generality.
\end{rem}
To conclude this part, we mention as an example that the adjoint representation
of $\mathfrak{sl}(m,\mathbb{C})$ is the irreducible representation
associated with the dominant integral functional $\lambda_{1}+\lambda_{m-1}=\mu_{1}-\mu_{m}.$

\begin{comment}
\\
\\
\textbf{II. An essential step: developing highest degree Lie
component into a Cartan subalgebra}\\
\textbf{}\\
\end{comment}

\paragraph*{II. An essential step: developing highest degree Lie
component into a Cartan subalgebra}
\label{sec: LowerII}
\addcontentsline{toc}{paragraph}{\nameref{sec: LowerII}}

$\ $\\
\\
Returning to our signature problem, let $\mathbf{X}_t=\exp(tl)$ be
a pure $m$-rough path, where $l\in\mathcal{L}^{(m)}(V)$ with highest
degree component $l_{m}.$ An essential step in our approach, is to choose
$\mathfrak{g}$ to be a finite dimensional complex semisimple Lie
algebra in the Lie algebraic development, together with a linear embedding
$F:V\rightarrow\mathfrak{g}$ such that the space ${\cal L}_{m}(V)$
of homogeneous Lie polynomials of degree $m$ is mapped into a Cartan
subalgebra of $\mathfrak{g}$ under the induced homomorphism on the
free Lie algebra ${\cal L}(V)$. In this way, according to Theorem
\ref{thm: IntLBdd}, we are immediately led to the lower estimate
\begin{equation}\label{eq: Weight Estimate}
L_{m}({\bf X})\geqslant\frac{\sup\left\{ {\rm Re}(\mu(F(l_{m}))):\mu\in\Pi(\rho)\right\} }{\|\Phi\|_{V\rightarrow\mathrm{End}(W)}^{m}}
\end{equation}
under the given Lie algebraic development $\Phi=\rho\circ F$, where recall
that $\Pi(\rho)\subseteq\mathfrak{h}^*$ is the set of weights for the representation $\rho$.
Representation theory then provides tractable methods
of computing weights for given representations, hence leading us
to more explicit lower bounds on $L_{m}({\bf X})$. 

The simplest way of mapping ${\cal L}_{m}(V)$ into a Cartan subalgebra
is through the Lie algebra $\mathfrak{sl}(m,\mathbb{C}),$ which can
be seen by straight forward matrix calculation. However, the essential
reason behind such calculation is hidden in the root pattern as stated
in the following lemma. Working with root patterns also allows one
to identify other semisimple Lie algebras which are not isomorphic
to $\mathfrak{sl}(m;\mathbb{C})$ but achieve the same property (cf.
Example \ref{exa: SO} and Example \ref{exa: G2} below).
\begin{lem}
\label{lem: IntoCartan}Suppose that there exist $m-1$ nonzero roots
$\alpha_{1},\cdots,\alpha_{m-1}$ with respect to $\mathfrak{h},$
such that all nonzero roots one can construct from them as integral
linear combinations are precisely of the form $\pm(\alpha_{i}+\alpha_{i+1}+\cdots+\alpha_{j})$
with $1\leqslant i\leqslant j\leqslant m-1.$ Define the subspace
\begin{equation}
E\triangleq\mathfrak{g}^{\alpha_{1}}\oplus\cdots\oplus\mathfrak{g}^{\alpha_{m-1}}\oplus\mathfrak{g}^{-(\alpha_{1}+\cdots+\alpha_{m-1})}.\label{eq: ESpace}
\end{equation}
Then 
\[
E^{(m-1)}\triangleq\underbrace{[\cdots[[[E,E],E],E]\cdots]}_{m-1\ {\rm brackets}}\subseteq\mathfrak{h}.
\]
\end{lem}
\begin{proof}
For each $1\leqslant i\leqslant m-1$ and $1\leqslant j\leqslant m-i,$
define 
\[
\alpha_{i;j}\triangleq\alpha_{i}+\alpha_{i+1}+\cdots+\alpha_{i+j-1}.
\]
According to the assumption, the $\alpha_{i;j}$'s are precisely all
the nonzero roots one can build from $\alpha_{1},\cdots,\alpha_{m-1}$
as integral linear combinations. It follows from the graded property
(\ref{eq: GradRoot}) of root spaces that 
\begin{align*}
E^{(1)} & =\left(\sum_{i=1}^{m-2}\mathfrak{g}^{\alpha_{i;2}}\right)\oplus\mathfrak{g}^{-\alpha_{1;m-2}}\oplus\mathfrak{g}^{-\alpha_{2;m-2}},\\
 & \cdots\\
E^{(k)} & =\left(\sum_{i=1}^{m-1-k}\mathfrak{g}^{\alpha_{i;k+1}}\right)\oplus\left(\sum_{j=1}^{k+1}\mathfrak{g}^{-\alpha_{j;m-1-k}}\right),\\
 & \cdots\\
E^{(m-2)} & =\mathfrak{g}^{\alpha_{1}+\cdots+\alpha_{m-1}}\oplus\left(\sum_{j=1}^{m-1}\mathfrak{g}^{-\alpha_{j}}\right).
\end{align*}
Finally, by using property (\ref{eq: GradRoot}) again as well as the assumption of the lemma, we obtain 
\[
E^{(m-1)}=[E^{(m-2)},E]\subseteq\mathfrak{\mathfrak{g}}^{0}=\mathfrak{h}.
\]
\end{proof}
Lemma \ref{lem: IntoCartan} tells us that, if we design $F:V\rightarrow\mathfrak{g}$
so that $F(V)\subseteq E$, then under the induced homomorphism on
the free Lie algebra, ${\cal L}_{m}(V)$ is mapped into the Cartan
subalgebra $\mathfrak{h}$.
\begin{example}
\label{exa: DevSL}Consider $\mathfrak{g}=\mathfrak{sl}(m,\mathbb{C}),$
with a Cartan subalgebra $\mathfrak{h}$ given by the subspace of
diagonal matrices in $\mathfrak{g}$. In this case, it is easy to
see that the simple roots $\alpha_{i}=\mu_{i}-\mu_{i+1}$ ($1\leqslant i\leqslant m-1$)
given by (\ref{eq: BaseSL}) satisfy the assumption of
Lemma \ref{lem: IntoCartan}. In this case, we have
\[
E=\left\{ \left(\begin{array}{cccc}
0 & z_{1} &  & 0\\
 &  & \ddots\\
 &  &  & z_{m-1}\\
z_{m} &  &  & 0
\end{array}\right):\ z_{1},\cdots,z_{m}\in\mathbb{C}\right\} ,
\]
where omitted entries in the matrix are all $0$'s. Indeed, the semisimple
Lie algebra associated with the root system generated by the roots given in Lemma \ref{lem: IntoCartan}
is isomorphic to $\mathfrak{sl}(m,\mathbb{C}).$
\end{example}
Using root patterns, we give two other examples that are not isomorphic
to $\mathfrak{sl}(m,\mathbb{C})$ but also allow one to map highest
degree Lie polynomials into a Cartan subalgebra. In each example,
the underlying Lie algebra is of rank two. The nonzero roots are drawn
as planar Euclidean vectors, in which integral linear combinations
follow usual vector operation rules. The corresponding conclusion is immediate by
manipulating the root vectors based on the graded property (\ref{eq: GradRoot})
and $\mathfrak{g}^{0}=\mathfrak{h}.$ Although possible, there is no need to work with
the actual Lie algebra $\mathfrak{g}$ and the associated root spaces at this level.
\begin{figure}\label{fig: B2G2} 
\begin{center} 
\includegraphics[scale=0.35]{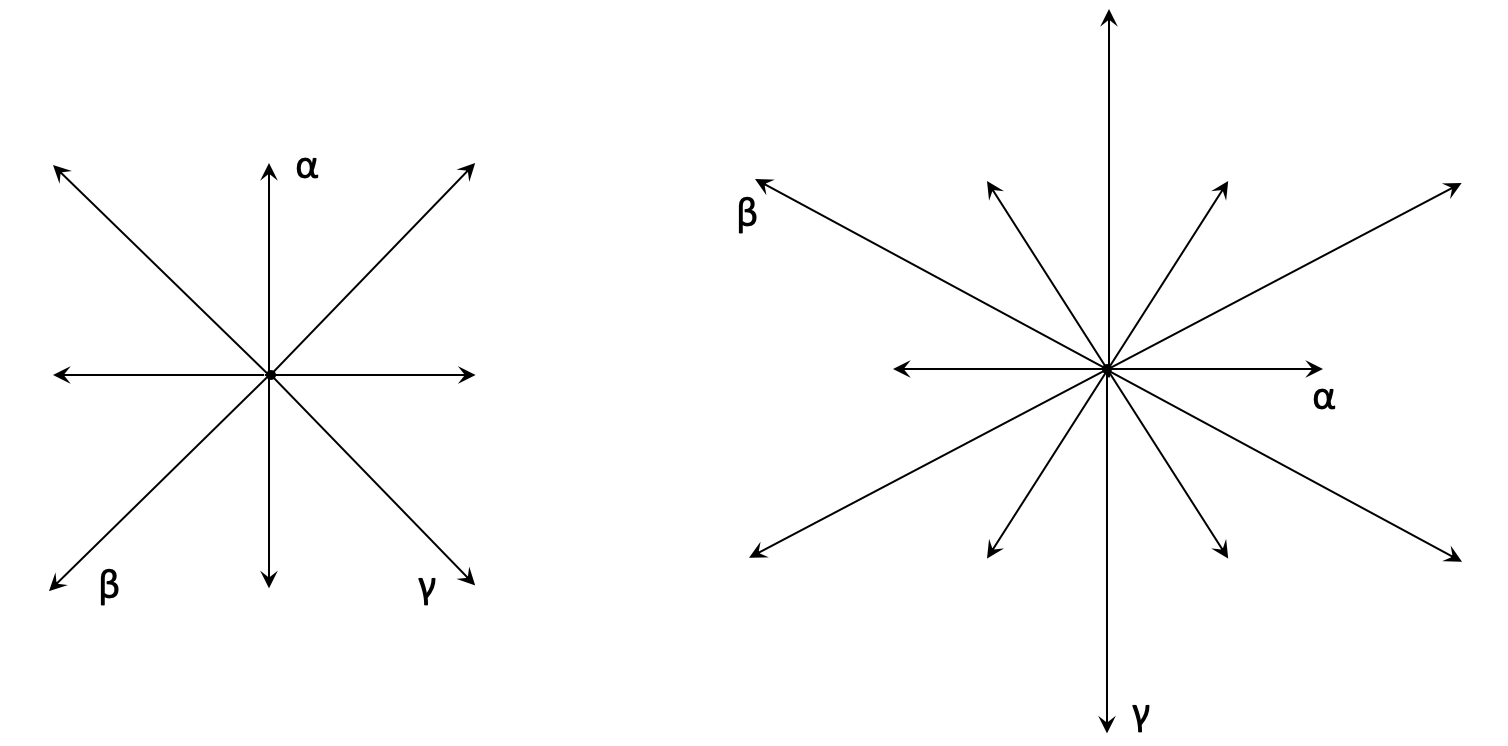}
\protect\caption{Root Systems of $\mathfrak{so}(5,\mathbb{C})$ and $\mathfrak{g}_2$}
\end{center}
\end{figure} 
\begin{example}
\label{exa: SO}Consider $\mathfrak{g}=\mathfrak{so}(5,\mathbb{C}),$
the Lie algebra of $5\times5$ complex skew-symmetric matrices. The
associated root system is given by the left hand side of Figure 4.1.
If we require $F:V\rightarrow E\triangleq\mathfrak{g}^{\alpha}\oplus\mathfrak{g}^{\beta}\oplus\mathfrak{g}^{\gamma}$,
then ${\cal L}_{4}(V)$ is mapped into a Cartan subalgebra (cf. Section
\ref{subsec: LowDeg} II below for more explicit calculations in degree
$m=4$ based on this structure). The property can be generalized to
higher degrees by considering $\mathfrak{so}(n,\mathbb{C})$ with
larger $n$.
\end{example}
\begin{example}
\label{exa: G2}Consider $\mathfrak{g}=\mathfrak{g}_{2},$ the smallest
exceptional simple Lie algebra. It arises from the classification
of simple Lie algebras, and can be identified as the Lie algebra of
the subgroup of ${\rm Spin}(7)$ preserving a point on $S^{7}$. The
associated root system is given by the right hand side of Figure 4.1.
If we require $F:V\rightarrow E\triangleq\mathfrak{g}^{\alpha}\oplus\mathfrak{g}^{\beta}\oplus\mathfrak{g}^{\gamma}$,
then ${\cal L}_{6}(V)$ is mapped into a Cartan subalgebra.
\end{example}
%

\begin{comment}
$\ $\\
\textbf{III. A consistency lemma for certain symmetric polynomial systems}\\
\textbf{}\\
\end{comment}

\paragraph*{III. A consistency lemma for certain symmetric polynomial systems}
\label{sec: LowerIII}
\addcontentsline{toc}{paragraph}{\nameref{sec: LowerIII}}

$\ $\\
\\
Note that the homogeneous Lie polynomial $l_{m}\in{\cal L}_{m}(V)$
has the general form $l_{m}=c_{1}h_{1}+\cdots+c_{\nu}h_{\nu},$ where
$\{h_{1},\cdots,h_{\nu}\}$ is a given basis of ${\cal L}_{m}(V)$.
In order to produce a lower estimate on $L_{m}({\bf X})$ in the form
of Theorem \ref{thm:MainThm}, with a factor \textit{independent of
the coefficients $c_{i}$'}s, a natural idea is to require each $h_{i}$
to have the right eigenvalue individually. In this way, properties of Cartan subalgebra will guarantee that $l_m$ has a desired eigenvalue $\|l_m\|$ and the operator
norm of the Lie algebraic development will depend only on the roughness
$m$ but not on the coefficients $c_{i}$. This viewpoint leads us
to the consideration of certain type of polynomial systems. 
A consistency lemma for these systems, stated as follows, will be needed for the proof of our main lower estimate.
The lemma may also be of independent interest.
\begin{lem}
\label{lem: Solvability}Let $p_{1},\cdots,p_{\nu}$ be homogeneous
polynomials over $\mathbb{C}^{n}$ of the same degree. Suppose that
they are linearly independent. Then there exists $k\geqslant1,$ such
that for any $c_{1},\cdots,c_{\nu}\in\mathbb{C},$ the polynomial system
\[
\begin{cases}
p_{1}({\bf z}_{1})+\cdots+p_{1}({\bf z}_{k})=c_{1},\\
\cdots\\
p_{\nu}({\bf z}_{1})+\cdots+p_{\nu}({\bf z}_{k})=c_{\nu}
\end{cases}
\]
has at least one solution in $\mathbb{C}^{kn}$, where ${\bf z}_{1},\cdots,{\bf z}_{k}$
are independent variables each having dimension $n$.
\end{lem}
\begin{rem}
Lemma \ref{lem: Solvability} is not as obvious as one may expect
and the special structure of the system has to play an essential role.
In general, a polynomial system in which the number of variables is greater than the number of equations may not always possess a solution, even when
assuming that the underlying polynomials are algebraically independent.
For instance, the system 
\[
x^{2}y=0,\ xyz=1
\]
does not have a solution! It is to some extent surprising that linear
independence is sufficient for the assertion to hold. 
\end{rem}
\begin{proof}[Proof of Lemma \ref{lem: Solvability}]\footnote{From the algebraic geometric viewpoint, it is not obvious how one can approach by using a general dimension argument, since in the associated projective space one needs to rule out the possibility that the underlying projective variety lies in the hyperplane at infinity. The proof we give here is entirely elementary.}
We treat the assertion as a property depending on $\nu$ (the number
of polynomials involved) and prove it by induction. When $\nu=1$, since
$p_{1}\neq0$ we know that $p_{1}({\bf z})\neq0$ for some ${\bf z}\in\mathbb{C}^{n}.$
Since $p_{1}$ is homogeneous, it follows from scaling that the image
of $p_{1}$ must be $\mathbb{C}.$ Therefore, the assertion holds
with $k=1$.

Suppose that the assertion is true for $\nu$ polynomials, and assume
that we are now given $\nu+1$ linearly independent homogeneous polynomials
$p_{1},\cdots,p_{\nu+1}$ of the same degree. By induction hypothesis,
there exists $k\geqslant1$, such that for any $1\leqslant i\leqslant \nu+1$,
the map ${\bf p}_{\hat{i}}^{(k)}:\mathbb{C}^{kn}\rightarrow\mathbb{C}^{\nu}$
defined by
\[
{\bf p}_{\hat{i}}^{(k)}({\bf z}_{1},\cdots,{\bf z}_{k})\triangleq\left(\begin{array}{c}
p_{1}({\bf z}_{1})+\cdots+p_{1}({\bf z}_{k})\\
\vdots\\
p_{i-1}({\bf z}_{1})+\cdots+p_{i-1}({\bf z}_{k})\\
p_{i+1}({\bf z}_{1})+\cdots+p_{i+1}({\bf z}_{k})\\
\vdots\\
p_{\nu+1}({\bf z}_{1})+\cdots+p_{\nu+1}({\bf z}_{k})
\end{array}\right)
\]
is surjective. We claim that, for every $1\leqslant i\leqslant \nu+1$,
the following system 
\begin{equation}
\begin{cases}
p_{i}({\bf z}_{1})+\cdots+p_{i}({\bf z}_{4k})\neq0,\\
p_{j}({\bf z}_{1})+\cdots+p_{j}({\bf z}_{4k})=0, & {\rm for\ all}\ j\neq i,
\end{cases}\label{eq: system with i-th equation nonzero}
\end{equation}
must have a solution. Observe that if this is true, then the induction
step finishes with $k'=4k(\nu+1).$ Indeed, let $c_{1},\cdots,c_{\nu+1}\in\mathbb{C}.$  For each $i$,
by homogenuity and scaling, the consistency of the
system (\ref{eq: system with i-th equation nonzero}) implies the
consistency of the system
\[
\begin{cases}
p_{1}({\bf z}_{1})+\cdots+p_{1}({\bf z}_{4k})=0,\\
\cdots\\
p_{i}({\bf z}_{1})+\cdots+p_{i}({\bf z}_{4k})=c_{i},\\
\cdots\\
p_{\nu+1}({\bf z}_{1})+\cdots+p_{\nu+1}({\bf z}_{4k})=0.
\end{cases}
\]
Let ${\bf Z}^{(i)}\in\mathbb{C}^{4kn}$ be a solution to the above
system. By adding up the $\nu+1$ cases, we know that the system 
\[
\begin{cases}
p_{1}({\bf z}_{1})+\cdots+p_{1}({\bf z}_{4k(\nu+1)})=c_{1},\\
\cdots\\
p_{\nu+1}({\bf z}_{1})+\cdots+p_{\nu+1}({\bf z}_{4k(\nu+1)})=c_{\nu+1},
\end{cases}
\]
has a solution given by ${\bf Z}=({\bf Z}^{(1)},\cdots,\mathbf{Z}^{(\nu+1)})\in\mathbb{C}^{4kn(\nu+1)}.$
In other words, the assertion holds with $k'=4k(\nu+1).$

Now it remains to show the consistency of the system (\ref{eq: system with i-th equation nonzero}).
Suppose on the contrary that the system is inconsistent for some $i$.
Without loss of generality, we may assume that $i=1$. We first introduce
some notation to simplify the presentation. It is convenient to call
\[
{\bf Z}=({\bf z}_{1},\cdots,{\bf z}_{k}),\ {\bf Z}'=({\bf z}_{k+1},\cdots,{\bf z}_{2k})
\]
and 
\[
{\bf W}=({\bf z}_{2k+1},\cdots,{\bf z}_{3k}),\ {\bf W}'=({\bf z}_{3k+1},\cdots,{\bf z}_{4k}).
\]
We also define
\[
P_{i}({\bf Z})=p_{i}({\bf z}_{1})+\cdots+p_{i}({\bf z}_{k})
\]
and similarly for the other parts of the variables. In particular,
we have
\[
p_{i}({\bf z}_{1})+\cdots+p_{i}({\bf z}_{4k})=P_{i}({\bf Z})+P_{i}({\bf Z}')+P_{i}({\bf W})+P_{i}({\bf W}').
\]
Under the above notation and assumption, we know that $P_{1}({\bf Z})+P_{1}({\bf Z}')+P_{1}({\bf W})+P_{1}({\bf W}')$
vanishes identically on the algebraic variety
\[
{\cal V}\triangleq\left\{ ({\bf Z},{\bf Z}',{\bf W},{\bf W}'):P_{i}({\bf Z})+P_{i}({\bf Z}')+P_{i}({\bf W})+P_{i}({\bf W}')=0\ {\rm for}\ 2\leqslant i\leqslant \nu+1\right\} 
\]
defined by the remaining polynomials.

We claim that, there exists a function $F:\mathbb{C}^{\nu}\rightarrow\mathbb{C},$
such that 
\begin{align}\label{eq: FRelation}
P_{1}({\bf W})+P_{1}({\bf W}') & =F\left(P_{2}({\bf W})+P_{2}({\bf W}'),\cdots,P_{\nu+1}({\bf W})+P_{\nu+1}({\bf W}')\right)
\end{align}
for all $({\bf W},{\bf W}')\in\mathbb{C}^{2kn}.$ Indeed, define $\Xi:\mathbb C^{2kn}\rightarrow \mathbb C^{\nu}$ by 
\begin{equation*}\Xi({\bf W},{\bf W}')\triangleq(P_2({\bf W})+P_2({\bf W}'),\dots,P_{\nu+1}({\bf W})+P_{\nu+1}({\bf W}')).\end{equation*} 
By the induction hypothesis, we know that $\Xi$ is surjective. We then define $F:\mathbb C^{\nu}\rightarrow \mathbb C$ by 
\begin{equation*}
F(\xi)\triangleq P_{1}({\bf W})+P_{1}({\bf W}'),
\end{equation*}where $(\mathbf{W},\mathbf{W}')$ is any element such that $\xi=\Xi(\mathbf{W},\mathbf{W}')$.
To verify
that $F$ is well defined, suppose that $\xi=\Xi({\bf W},{\bf W}')=\Xi(\tilde{\bf W},\tilde{\bf W}').$  Then
\[
P_{j}({\bf W})+P_{j}({\bf W}')=P_{j}(\tilde{{\bf W}})+P_{j}(\tilde{{\bf W}}'),\ \ \ {\rm for\ all}\ 2\leqslant j\leqslant \nu+1.
\]
Let 
\[
({\bf Z},{\bf Z}')\triangleq(-1)^{1/m}\cdot({\bf W},{\bf W}'),
\]
where $m$ is the degree of the underlying polynomials. It follows
that both of $({\bf Z},{\bf Z}',{\bf W},{\bf W}')$ and $({\bf Z},{\bf Z}',\tilde{{\bf W}},\tilde{{\bf W}}')$
are elements in ${\cal V},$ and thus they are both zeros of the polynomial
at $i=1$. In particular, we have 
\[
P_{1}({\bf W})+P_{1}({\bf W}')=P_{1}(\tilde{{\bf W}})+P_{1}(\tilde{{\bf W}}'),
\]
showing that $F$ is well defined.

By taking ${\bf W}'=0$ in (\ref{eq: FRelation}), we arrive at 
\[
P_{1}({\bf W})=F(P_{2}({\bf W}),\cdots P_{\nu+1}({\bf W})).
\]
Now the key observation is that, $F$ must be linear. Indeed, given
$\lambda\in\mathbb{C},$ we have 
\begin{align*}
 & \lambda F\left(P_{2}({\bf W}),\cdots,P_{\nu+1}({\bf W})\right)\\
 & =\lambda P_{1}({\bf W})\\
 & =P_{1}(\lambda^{1/m}{\bf W})\\
 & =F\left(P_{2}(\lambda^{1/m}{\bf W}),\cdots,P_{\nu+1}(\lambda^{1/m}{\bf W})\right)\\
 & =F\left(\lambda P_{2}({\bf W}),\cdots,\lambda P_{2}({\bf W})\right).
\end{align*}
In addition, let $\xi,\eta\in\mathbb{C}^{\nu}.$ Again by induction hypothesis,
there exist ${\bf W}$ and ${\bf W}'$ in $\mathbb{C}^{kn}$, such
that 
\[
\xi=\left(P_{2}({\bf W}),\cdots,P_{\nu+1}({\bf W})\right),\ \eta=\left(P_{2}({\bf W}'),\cdots,P_{\nu+1}({\bf W}')\right).
\]
 It follows that
\begin{align*}
F(\xi+\eta) & =F\left(P_{2}({\bf W})+P_{2}({\bf W}'),\cdots,P_{\nu+1}({\bf W})+P_{\nu+1}({\bf W}')\right)\\
 & =P_{1}({\bf W})+P_{1}({\bf W}')\\
 & =F(\xi)+F(\eta).
\end{align*}
Therefore, $F$ is linear. This leads to a contradiction with the
linear independence among $P_{1},\cdots,P_{\nu+1}$. Consequently, the
system (\ref{eq: system with i-th equation nonzero}) is consistent,
finishing the proof of the induction step.
\end{proof}
\begin{rem}
\label{rem: ExplicitK}It is seen from the inductive argument in the proof that one can take
$k=4^{\nu-1}\nu!$ in the lemma. This observation will be used to produce
an explicit estimate on the factor $c(m,d)$ arising from the main
lower bound (cf. Theorem \ref{thm: ExpC} below).
\end{rem}
%
\begin{comment}
$\ $\\
\textbf{IV. Establishing the main lower estimate}\\
\textbf{}\\
\end{comment}

\paragraph*{IV. Establishing the main lower estimate}
\label{sec: LowerIV}
\addcontentsline{toc}{paragraph}{\nameref{sec: LowerIV}}

$\ $\\
\\
Using the representation theory of $\mathfrak{sl}(n;\mathbb{C})$
and Lemma \ref{lem: Solvability}, we can now establish our main lower
estimate for the signature tail asymptotics of pure rough paths. The
result contains two parts, a general lower estimate involving a multiplicative
factor $c(m,d)$, and an explicit estimate on this factor. We state
and prove them separately.

First of all, our general lower estimate is stated as follows. The
proof is based on designing appropriate Lie algebraic developments.
\begin{thm}
\label{thm: MainLBdd}Let $V$ be a $d$-dimensional Banach
space and let every tensor product $V^{\otimes n}$ be equipped with the associated projective tensor norm. For each $m\geqslant1,$ there exists a constant $c(m,d)\in(0,1]$
depending only on $m$ and $d$, such that 
\[
L_{m}({\bf X})\geqslant c(m,d)\|\pi_{m}(l)\|
\]
for all pure $m$-rough paths ${\bf X}_{t}=\exp(tl)\in G^{(m)}(V)$ over $V.$
\end{thm}
\begin{proof}
We write the highest degree component of $l$ in the form $l_{m}=c_{1}h_{1}+\cdots+c_{\nu}h_{\nu}$,
where $\{h_{1},\cdots,h_{\nu}\}$ is a given basis of ${\cal L}_{m}(V)$.
Using the dual characterization (\ref{eq: DualProj}) of the projective tensor norm, let $B$ be a given $m$-linear
functional over $V$ whose norm is bounded by $1$. We aim at constructing
a Lie algebraic development $\Phi:V\rightarrow\mathfrak{g}\rightarrow{\rm End}(W)$
such that:\\
\\
(i) $\mathfrak{g}$ is semisimple, and the space $\mathcal{L}_m(V)$ is mapped into a Cartan subalgebra $\mathfrak{h}$ of $\mathfrak{g}$ under the Lie homomorphism induced by $F$;\\
(ii) there exists a weight $\mu\in\mathfrak{h}^*$ for $\rho$ such that $\mu(F(l_m))=B(l_m)$;\\
(iii) the operator norm of $\Phi$ is bounded above by a constant which
is independent of $B$ and the specific values of the coefficients
$c_{i}$.\\
\\
If this can be achieved, the general lower estimate will follow from
(\ref{eq: Weight Estimate}) and (\ref{eq: DualProj}) since $B$ is
arbitrary.

One way of constructing such a development is the following. For simplicity
we assume that $\dim V=2$ with a given basis $\{{\rm e}_{1},{\rm e}_{2}\}$
(there is only notational difference in higher dimensions). We choose
$\mathfrak{g}=\mathfrak{sl}(k\cdot m,\mathbb{C})$ where $k\geqslant1$
is a large number to be specified later on. We choose a Cartan subalgebra
$\mathfrak{h}$ and a base of simple roots according to the discussion in Part I of
the current section. Define the embedding $F:V\rightarrow\mathfrak{g}$
in the following block diagonal form
\[
F({\rm e}_{1})=\left(\begin{array}{cccc}
A_{1} &  &  & 0\\
 & A_{2}\\
 &  & \ddots\\
0 &  &  & A_{k}
\end{array}\right)_{km\times km},\ F({\rm e}_{2})=\left(\begin{array}{cccc}
B_{1} &  &  & 0\\
 & B_{2}\\
 &  & \ddots\\
0 &  &  & B_{k}
\end{array}\right)_{km\times km},
\]
where each $A_{i},B_{j}\in\mathfrak{sl}(m,\mathbb{C})$ ($1\leqslant i,j\leqslant k$)
has the form 
\[
A_{i}=\left(\begin{array}{cccc}
0 & a_{i,1} &  & 0\\
 &  & \ddots\\
 &  &  & a_{i,m-1}\\
a_{i,m} &  &  & 0
\end{array}\right)_{m\times m},\ B_{j}=\left(\begin{array}{cccc}
0 & b_{j,1} &  & 0\\
 &  & \ddots\\
 &  &  & b_{j,m-1}\\
b_{j,m} &  &  & 0
\end{array}\right)_{m\times m},
\]
with all the $a_{i,r},b_{j,s}$'s being complex parameters to be specified
later on. There are totally $2km$ independent variables to determine
$F$. According to Lemma \ref{lem: IntoCartan} and Example \ref{exa: DevSL},
under the induced homomorphism (still denoted as $F$) on the free
Lie algebra, ${\cal L}_{m}(V)$ is mapped into the given Cartan subalgebra
$\mathfrak{h}$.

Finally, according to Theorem \ref{thm: ClasRep}, we choose $\rho:\mathfrak{g}\rightarrow{\rm End}(W)$
to be the irreducible representation of $\mathfrak{g}$ associated
with the $k$-th fundamental dominant integral functional $\lambda_{k}$,
and more explicitly by Theorem \ref{thm: RepSL} in the $\mathfrak{sl}(n,\mathbb{C})$
case, we have $W=\Lambda^k(\mathbb{C}^{km})$ and $\rho$ being
the $k$-th exterior power of the canonical matrix representation.
According to the same theorem, a weight for this representation is
given by 
\[
\mu=\mu_{1}+\mu_{m+1}+\mu_{2m+1}+\cdots+\mu_{(k-1)m+1}\in\mathfrak{h}^{*},
\]
where recall that $\mu_{i}$ is the linear functional of taking the
$i$-th diagonal entry.

To specify the parameters in order to fulfil the eigenvalue condition
(ii) while respecting the uniformity condition (iii), we are led to
setting up a system of equations:
\[
\mu(F(h_{i}))=B(h_{i}),\ 1\leqslant i\leqslant\nu.
\]
This is a polynomial system with $\nu$ equations and $2km$ independent
complex variables. It has the form 
\begin{equation}
\begin{cases}
p_{1}(A_{1},B_{1})+\cdots+p_{1}(A_{k},B_{k})=B(h_{1}),\\
\cdots\\
p_{\nu}(A_{1},B_{1})+\cdots+p_{\nu}(A_{k},B_{k})=B(h_{\nu}),
\end{cases}\label{eq: MainSyst}
\end{equation}
where each $p_{i}$ is a homogeneous polynomial of degree $m$ in
$2m$ complex variables. More precisely, $p_{i}(A,B)$ is the first
entry of the diagonal polynomial matrix $G(h_{i}),$ where $G$ is
the homomorphism induced from the linear map $V\rightarrow\mathfrak{sl}(m,\mathbb{C}[a_{i},b_{j}])$
given by 
\[
{\rm e}_{1}\mapsto A\triangleq\left(\begin{array}{cccc}
0 & a_{1} &  & 0\\
 &  & \ddots\\
 &  &  & a_{m-1}\\
a_{m} &  &  & 0
\end{array}\right),\ {\rm e}_{2}\mapsto B\triangleq\left(\begin{array}{cccc}
0 & b_{1} &  & 0\\
 &  & \ddots\\
 &  &  & b_{m-1}\\
b_{m} &  &  & 0
\end{array}\right).
\]
It is important to view $G$ as a homomorphism into the polynomial
matrix algebra in $2m$ complex variables.

We claim that, the polynomial system (\ref{eq: MainSyst}) has a solution
in $\mathbb{C}^{2km}$ for some large $k\geqslant1$, which according
to Lemma \ref{lem: Solvability}, boils down to showing that the polynomials
$p_{1},\cdots,p_{\nu}\in\mathbb{C}[a_{i},b_{j}]$ are linearly independent.
To this end, consider the linear map $T:V^{\otimes m}\rightarrow\mathbb{C}[a_{i},b_{j}]$
defined by
\[
T({\rm e}_{i_{1}}\otimes\cdots\otimes{\rm e}_{i_{m}})\triangleq(1,1)\text{-entry}{\rm \ of}\ G({\rm e}_{i_{1}}\otimes\cdots\otimes{\rm e}_{i_{m}}).
\]
Explicit calculation then shows that 
\begin{equation}
T({\rm e}_{i_{1}}\otimes\cdots\otimes{\rm e}_{i_{m}})=w_{i_{1}}\cdots w_{i_{m}},\label{eq: Monomial}
\end{equation}
where $w_{i_{j}}=a_{j}$ or $b_{j}$ according to whether $i_{j}=1$
or $2$. In particular, we see that $T$ is injective. Since $h_{1},\cdots,h_{\nu}$
is a basis of ${\cal L}_{m}(V)\subseteq V^{\otimes m},$ we conclude
that the polynomials
\[
p_{i}(A,B)=T(h_{i}),\ 1\leqslant i\leqslant\nu
\]
are linearly independent. Therefore, by Lemma \ref{lem: Solvability}, the polynomial system
(\ref{eq: MainSyst}) has a solution for some large $k$. Any solution
can be used to determine the Lie algebraic development $\Phi=\rho\circ F$
specified in the previously given form. Under such development, it
follows from Theorem \ref{thm: IntLBdd} that 
\[
L_{m}({\bf X})\geqslant\frac{B(l_{m})}{\|\Phi\|_{{V\rightarrow\mathrm{End}(W)}}^{m}}.
\]

Now it remains to estimate the operator norm of $\Phi$, which reduces to
estimating a solution to the polynomial system (\ref{eq: MainSyst}).
For this purpose, according to Lemma \ref{lem: Solvability}, there
exists $k\geqslant1$, such that for each $1\leqslant i\leqslant\nu$,
the polynomial system
\begin{equation}
\begin{cases}
p_{i}(A_{1},B_{1})+\cdots+p_{i}(A_{k},B_{k})=1,\\
p_{j}(A_{1},B_{1})+\cdots+p_{j}(A_{k},B_{k})=0, & j\neq i,
\end{cases}\label{eq: NmdSyst}
\end{equation}
has a solution ${\bf Z}^{(i)}\in\mathbb{C}^{2km}.$ It follows that
with $\tilde{{\bf Z}}^{(i)}\triangleq B(h_{i})^{1/m}{\bf Z}^{(i)},$
the vector $\tilde{{\bf Z}}\triangleq(\tilde{{\bf Z}}^{(1)},\cdots,\tilde{{\bf Z}}^{(\nu)})\in\mathbb{C}^{2k\nu m}$
is a solution to the system (\ref{eq: MainSyst}) with $k$ being
enlarged to $k\nu.$ Since $\|B\|\leqslant1,$ we see that $\tilde{{\bf Z}}$,
and thus the operator norm of $\Phi$, is bounded above by a constant depending
only on the roughness $m$ and the dimension $d$. Since $B$ is arbitrary, this implies the desired
lower estimate with a multiplicative factor $c(m,d)$ depending only on $m$ and $d$.
\end{proof}
It is clear from the last paragraph of the previous proof that, the
key to estimating the multiplicative factor
$c(m,d)$ is an explicit estimate on a solution to the system (\ref{eq: MainSyst}).
In general, selecting solutions to a consistent polynomial system
with a priori bounds is an important topic in computational algebraic
geometry that has been studied by many authors. We state a result
of Vorob'ev \cite{Vorobev86} that is relevant to us. Recall that
the \textit{bitsize} of a nonzero integer $n$ is the unique natural
number $\tau$ such that $2^{\tau-1}\leqslant |n|<2^{\tau}.$ The bitsize
of a rational number is the sum of the bitsizes of its numerator and
denominator.

\begin{lem}[cf. \cite{Vorobev86}, Theorem 3]\label{lem: Vorobev}

Let ${\cal V}$ be the set of real solutions to a consistent system
of polynomial equations $f_{1}=\cdots=f_{\nu}=0$ where each $f_{i}\in\mathbb{Q}[x_{1},\cdots,x_{n}].$
Let $L$ be the maximum of the bitsizes of the coefficients of the
system, $D\triangleq\sum_{i=1}^{\nu}\deg f_{i}$ and $r\triangleq\left(\begin{array}{c}
n+2D\\
n
\end{array}\right)$. Then there exists a point $x=(x_{1},\cdots,x_{n})\in{\cal V},$
such that 
\[
|x_{i}|\leqslant2^{H(r,L)}\ \ \ {\rm for\ all}\ 1\leqslant i\leqslant n,
\]
where $H$ is some universal bivariate polynomial independent of the
original system.

\end{lem}
\begin{rem}
In Vorob'ev's result (and other results of similar type), having rational
or sometimes  integral coefficients is a crucial assumption. In
addition, it presumes the consistency of the system before locating
an a priori bounded solution. In particular, it does not provide
a proof on whether the system admits a solution.
\end{rem}
With the help of Vorob'ev's estimate, we can now establish an explicit
estimate on the factor $c(m,d)$ arising from Theorem \ref{thm: MainLBdd}.
\begin{thm}
\label{thm: ExpC}Keeping the same notation as in Theorem \ref{thm: MainLBdd},
the multiplicative factor $c(m,d)$ satisfies
\[
c(m,d)\geqslant \Lambda_{d}^{-m}\cdot2^{-(\nu_{m,d}!)^{\gamma\nu_{m,d}}},
\]
where $\Lambda_{d}$ is a constant depending only on $d$, $\nu_{m,d}\triangleq\dim{\cal L}_{m}(V)$,
and $\gamma>1$ is a universal constant.
\end{thm}
\begin{proof}
Essentially we just need to keep track of the quantities appearing
in the proof of Theorem \ref{thm: MainLBdd} in a precise way.

First of all, in that proof we fix a basis $\{{\rm e}_{1},\cdots,{\rm e}_{d}\}$
of $V$ with norm $1$, and assume that $\{h_{1},\cdots,h_{\nu}\}$
is a Hall basis of ${\cal L}_{m}(V)$ built over the letters ${\rm e}_{1},\cdots,{\rm e}_{d}$.
Next, in the representation $\rho:\mathfrak{sl}(k\cdot m,\mathbb{C})\rightarrow\Lambda^k(\mathbb{C}^{km})$,
we work with the $l^{1}$-norm on $\Lambda^k(\mathbb{C}^{km})$
with respect to the canonical exterior basis. In addition, by Remark
\ref{rem: ExplicitK} we choose $k=4^{\nu-1}\nu!$ for the system
(\ref{eq: NmdSyst}). Recall that ${\bf Z}^{(i)}$ (respectively,
$\tilde{{\bf Z}}$) is a solution to the system (\ref{eq: NmdSyst})
(respectively, (\ref{eq: MainSyst})). Now we presume that for each
$i,$ all components of ${\bf Z}^{(i)}$ are bounded by a number $M$.
Using the observation that $\|h_{i}\|\leqslant2^{m},$ we know that
all components of $\tilde{{\bf Z}}$ are bounded by $2M$. It then
follows from a simple unwinding of definitions that 
\begin{equation}
\|\Phi\|_{V\rightarrow\mathrm{End}(W)}\leqslant2\Lambda_{d}kM,\label{eq: OPfromSol}
\end{equation}
where $\Lambda_{d}$ is the constant depending only on $d$ which arises
from the comparison between the given norm $\|\cdot\|_{V}$ on $V$
and the $l^{1}$-norm $\|\cdot\|_{1}$ with respect to the basis $\{{\rm e}_{1},\cdots,{\rm e}_{d}\}$,
i.e. $\|\cdot\|_{1}\leqslant \Lambda_{d}\|\cdot\|_{V}.$

It remains to work out $M$ explicitly. The first observation is that,
the system (\ref{eq: NmdSyst}) has integral coefficients each being
bounded by $2^{m}.$ To apply Vorob'ev's estimate, we need to turn
the system into an equivalent system over real variables, which can
be done by viewing each complex variable as a pair of real variables.
In this way, (\ref{eq: NmdSyst}) becomes a system with $2\nu$ equations
and $2kdm$ real variables. The next observation is that, the new
system again has integer coefficients, and more importantly when transforming
from complex to real, the coefficients are not enlarged. This is due
to the fact that the polynomial $p_{i}$ is linear with respect to
every single complex variable when the others are frozen (cf. (\ref{eq: Monomial})
for the shape of relevant monomials). Therefore, using the notation
in Lemma \ref{lem: Vorobev}, we find that 
\[
L\leqslant m,\ D=2m\nu,\ n=\frac{1}{2}dm4^{\nu}\nu!,\ r=\left(\begin{array}{c}
n+2D\\
n
\end{array}\right).
\]
It follows from Stirling's approximation and Vorob'ev's estimate that
$M\leqslant2^{(\nu!)^{\kappa\nu}}$ with some universal constant $\kappa>1$
independent of the system. Now the result follows by substituting
this into (\ref{eq: OPfromSol}) and using Theorem \ref{thm: MainLBdd}.
\end{proof}
\begin{rem}
The proof of Theorem \ref{thm: MainLBdd} does not provide the optimal
way of constructing the Lie algebraic development $\Phi$ in general,
and the explicit lower bound given by Theorem \ref{thm: ExpC} does
not seem to be optimal either. To improve the estimate, among the
class of Lie algebraic developments $\Phi$ in which $\|\pi_{m}(l)\|$
is an eigenvalue of $\Phi(\pi_{m}(l)),$ one needs to minimize the
operator norm of $\Phi.$ As we will see in low degree cases, there
are plenty of rooms for reducing the operator norm of $\Phi$ and
hence improving the factor $c(m,d)$. The sharp lower bound (Conjecture
\ref{conj: LengConjPure}) will hold if one can achieve $\|\Phi\|_{V\rightarrow\mathrm{End}(W)}=1$.
\end{rem}

As an immediate corollary of our methodology, we prove the following separation of points property for signatures. Such a separation property was first obtained by Chevyrev-Lyons \cite{CL16} 
as a key ingredient of proving their uniqueness result for the expected signature of stochastic processes.

\begin{cor}\label{cor: SepProp}  Let $V$ be a finite dimensional vector space. \\
 (1) Let $l,l'\in \mathcal{L}(V)$ be two distinct Lie polynomials over $V$. Then there exists a finite dimensional development $\Phi:V\rightarrow\mathrm{End}(W)$ such that $\Phi(l)\neq\Phi(l')$. \\
(2) Let $g_1,g_2$ be the signatures of two weakly geometric rough paths over $V$. Suppose that $g_1\neq g_2$. Then there exists a finite dimensional development $\Phi:V\rightarrow\mathrm{End}(W)$ such that $\Phi(g_1)\neq\Phi(g_2)$.
\end{cor}

\begin{proof}
(1) Let $m\geqslant 1$ be the smallest integer such that $\pi_m(l)\neq\pi_m(l')$. According to the proof of Theorem \ref{thm: MainLBdd}, there exists a finite dimensional Lie algebraic development 
\[
\Phi:V\stackrel{F}{\longrightarrow}\mathfrak{g}\stackrel{\rho}{\longrightarrow}\mathrm{End}(W)
\]such that \[
\Phi(\pi_{m}(l))\neq\Phi(\pi_{m}(l')).
\]More explicitly, we have $\mathfrak{g}=\mathfrak{sl}(k\cdot m,\mathbb{C})$ and $W=\Lambda^k(\mathbb{C}^{km})$ with $k=4^{\nu-1}\nu!$ and $\nu=\dim \mathcal{L}_m(V)$. For given $\varepsilon>0$, define $\Phi_\varepsilon\triangleq \rho\circ(\varepsilon\cdot F)$. It follows that
\begin{align*}
\Phi_{\varepsilon}(l-l') & =(\rho\circ(\varepsilon\cdot F))(l-l')\\
 & =\rho\left(\varepsilon^{m}\cdot F(\pi_{m}(l-l'))+\sum_{n>m}\varepsilon^{n}\cdot F(\pi_{n}(l-l'))\right)\\
 & =\varepsilon^{m}\cdot\Phi(\pi_{m}(l-l'))+\sum_{n>m}\varepsilon^{n}\cdot\Phi(\pi_{n}(l-l')).
\end{align*}Note that the summation is indeed finite since $l,l'$ are Lie polynomials. Therefore, we see that\[
\Phi_{\varepsilon}(l-l')=\varepsilon^{m}\cdot\Phi(\pi_{m}(l-l'))+o(\varepsilon^{m}),
\]which implies that $\Phi_\varepsilon(l-l')\neq0$ when $\varepsilon$ is small. Any such $\Phi_\varepsilon$ will satisfy the desired property.

(2) Write $g=\exp(l)$ and $g'=\exp(l')$ where $l,l'$ are Lie series respectively. In the same way as the proof of the first part, let $m\geqslant1$ be the smallest integer such that $\pi_m(l)\neq\pi_m(l')$, and choose a finite dimensional Lie algebraic development $\Phi=\rho\circ F:V\rightarrow \mathfrak{g}\rightarrow \mathrm{End}(W)$ separating $\pi_m(l)$ and $\pi_m(l')$. Since $g$ and $g'$ are path signatures, it is known that (cf. \cite{LS06} and \cite{Chevyrev13}), $l$ and $l'$ both have positive radius of convergence when viewed as formal tensor series. In particular, both of \[
\varepsilon\mapsto\Phi_{\varepsilon}(l),\ \varepsilon\mapsto\Phi_{\varepsilon}(l')
\]
are analytic functions in some neighbourhood of $\varepsilon=0$ where $\Phi_\varepsilon\triangleq \rho\circ(\varepsilon\cdot F)$. Therefore, we see that\[
\Phi_{\varepsilon}(l)=\varepsilon^{m}\cdot\Phi(\pi_{m}(l))+o(\varepsilon^{m}),\ \ \ \Phi_{\varepsilon}(l')=\varepsilon^{m}\cdot\Phi(\pi_{m}(l'))+o(\varepsilon^{m}),
\]when $\varepsilon$ is small. Note that we also have \[
\Phi_{\varepsilon}(g)=\exp\left(\Phi_{\varepsilon}(l)\right),\ \ \ \Phi_{\varepsilon}(g')=\exp\left(\Phi_{\varepsilon}(l')\right).
\]Since the exponential map for the group $\mathrm{Aut}(W)$ is a local diffeomorphism at the identity, the desired separation property holds under the development $\Phi_\varepsilon$ when $\varepsilon$ is small.
\end{proof}

\begin{rem}
One advantage of stating the separation property at the level of free Lie algebra is that the property becomes purely algebraic. Even at the level of  signature, the dependence on analytic properties is rather mild. Indeed, the proof of the positive radius of convergence for the logarithmic signature given in \cite{Chevyrev13} requires only the faster-than-geometric decay for signature components. This is the only analytic condition needed here.
\end{rem}

\subsubsection{\label{subsec: LowDeg}Explicit calculations in low degrees}

We perform some more explicit calculations in low degrees to illustrate
the methodology better. We consider $V=\mathbb{R}^{2}$ equipped with
the $l^{1}$-norm with respect to the standard basis $\{{\rm e}_{1},{\rm e}_{2}\}$. The associated projective tensor norm then coincides with the $l^1$-norm with respect to the canonical tensor basis.
In this context, we are going to show
that, the sharp lower bound holds in degrees $m=2,3$ and some cases in degrees $m=4,5$. When $m=4$, we have $c(4,2)\geqslant5/32$ in general.

\begin{comment}
\\
\\
\textbf{I. Sharp lower bound in degrees $2$ and $3$}\\
\\
\end{comment}

\paragraph*{I. Sharp lower bound in degrees $2$ and $3$}
\label{sec: ExpI}
\addcontentsline{toc}{paragraph}{\nameref{sec: ExpI}}

$\ $\\
\\
Let $\mathbf{X}_{t}=\exp(tl)$ be a pure $2$-rough path, and write
$\pi_{2}(l)=c[{\rm e}_{1},{\rm e}_{2}]\in{\cal L}_{2}(V)$. In order
to develop ${\cal L}_{2}(V)$ into a Cartan subalgebra, according
to Lemma \ref{lem: IntoCartan} and Example \ref{exa: DevSL}, we
choose $\mathfrak{g}=\mathfrak{sl}(2,\mathbb{C})$, and define $F:V\rightarrow\mathfrak{g}$
by 
\[
F({\rm e}_{1})=\left(\begin{array}{cc}
0 & a_{1}\\
a_{2} & 0
\end{array}\right),\ F({\rm e}_{2})=\left(\begin{array}{cc}
0 & b_{1}\\
b_{2} & 0
\end{array}\right),
\]
where $a_{1},a_{2},b_{1},b_{2}$ are parameters to be specified. In
addition, we choose $\rho:\mathfrak{g}\rightarrow{\rm End}(\mathbb{C}^{2})$
to be the canonical matrix representation, where $\mathbb{C}^{2}$
is equipped with the standard $l^{1}$-norm.

Note that 
\[
F([{\rm e}_{1},{\rm e}_{2}])=\left(\begin{array}{cc}
a_{1}b_{2}-a_{2}b_{1} & 0\\
0 & a_{2}b_{1}-a_{1}b_{2}
\end{array}\right)\in\mathfrak{h}.
\]
Since $\|\pi_{2}(l)\|=2|c|$, we set up the equation
\begin{equation}
a_{1}b_{2}-a_{2}b_{1}=+2\ {\rm or}\ -2,\label{eq: Deg2Eq}
\end{equation}
depending on whether $c$ is positive or negative. This will allow
us to produce $\|\pi_{2}(l)\|$ as an eigenvalue of $\Phi(\pi_{2}(l))\in{\rm End}(\mathbb{C}^{2}).$
Among all solutions, the minimum $\|\Phi\|_{\mathbb{R}^2\rightarrow\mathrm{End}(\mathbb{C}^2)}=1$ is obtained
at 
\[
a_{1}=a_{2}=1,\ b_{1}=\mp1,\ b_{2}=\pm1,
\]
where the signs are chosen depending on whether $c$ is positive or
negative. According to Theorem \ref{thm: upper} and Theorem \ref{thm: IntLBdd},
we conclude that $L_{2}({\bf X})=\|\pi_{2}(l)\|$ and thus Conjecture
\ref{conj: LengConjPure} holds for roughness $m=2$\textbf{.}

Next we consider the case when $l\in{\cal L}^{(3)}(V)$. In this case,
$\pi_{3}(l)\in{\cal L}_{3}(V)$ takes the form 
\[
\pi_{3}(l)=c_{1}[{\rm e}_{1},[{\rm e}_{1},{\rm e}_{2}]]+c_{2}[[{\rm e}_{1},{\rm e}_{2}],{\rm e}_{2}].
\]
To develop ${\cal L}_{3}(V)$ into a Cartan subalgebra, we choose
$\mathfrak{g}={\rm \mathfrak{sl}(3,\mathbb{C})},$ define $F:V\rightarrow\mathfrak{g}$
by 
\[
F({\rm e}_{1})=\left(\begin{array}{ccc}
0 & a_{1} & 0\\
0 & 0 & a_{2}\\
a_{3} & 0 & 0
\end{array}\right),\ F({\rm e}_{2})=\left(\begin{array}{ccc}
0 & b_{1} & 0\\
0 & 0 & b_{2}\\
b_{3} & 0 & 0
\end{array}\right)
\]
where $a_{i},b_{j}$'s are parameters to be determined, and choose
$\rho:\mathfrak{g}\rightarrow{\rm End}(\mathbb{C}^{3})$ to be the
canonical matrix representation where $\mathbb{C}^{3}$ is equipped
with the standard $l^{1}$-norm.

Suppose that $c_{1},c_{2}>0,$ under which we have $\|\pi_{3}(l)\|=4c_{1}+4c_{2}$.
To match the eigenvalues, we set up a system of equations
\[
\mu_{1}(F([{\rm e}_{1},[{\rm e}_{1},{\rm e}_{2}]]))=4,\ \mu_{1}(F([[{\rm e}_{1},{\rm e}_{2}],{\rm e}_{2}]))=4,
\]
where recall that $\mu_{1}$ is a weight for $\rho$ defined by taking
the first diagonal entry. By direct calculation, the system reads
\[
\begin{cases}
a_{1}a_{2}b_{3}+a_{2}a_{3}b_{1}-2a_{1}a_{3}b_{2}=4,\\
a_{1}b_{2}b_{3}-2a_{2}b_{1}b_{3}+a_{3}b_{1}b_{2}=4.
\end{cases}
\]
Among all its solutions, the minimum $\|\Phi\|_{\mathbb{R}^2\rightarrow\mathrm{End}(\mathbb{C}^3)}=1$ is achieved
at 
\[
a_{1}=a_{2}=1,\ a_{3}=-1,\ b_{1}=-1,\ b_{2}=b_{3}=1.
\]
The cases for other sign conditions on $c_{1},c_{2}$ are treated
similarly. Therefore, Conjecture \ref{conj: LengConjPure} holds for
roughness $m=3.$

\begin{comment}
\\
\\
\textbf{II. The degree $4$ case}\\
\\
\end{comment}

\paragraph*{II. The degree $4$ case}
\label{sec: ExpII}
\addcontentsline{toc}{paragraph}{\nameref{sec: ExpII}}

$\ $\\
\\
Now consider $l\in{\cal L}^{(4)}(V)$ with $\pi_{4}(l)=c_{1}h_{1}+c_{2}h_{2}+c_{3}h_{3},$
where 
\[
h_{1}=[[{\rm e}_{1},[{\rm e}_{1},{\rm e}_{2}]],{\rm e}_{1}],\ h_{2}=[[[{\rm e}_{1},{\rm e}_{2}],{\rm e}_{2}],{\rm e}_{2}],\ h_{3}=[{\rm e}_{1},[[{\rm e}_{1},{\rm e}_{2}],{\rm e}_{2}]]
\]
form a Hall basis of ${\cal L}_{4}(V).$ In this case, we demonstrate
the possibility of using other root systems that are not isomorphic
to $\mathfrak{sl}(n,\mathbb{C}),$ and show that 
\begin{equation}
L_{4}({\bf X})\geqslant\begin{cases}
\frac{5}{32}\|\pi_{4}(l)\|, & c_{1}\cdot c_{2}\geqslant0,\\
\frac{\sqrt{7}}{8}\|\pi_{4}(l)\|, & c_{1}\cdot c_{2}<0.
\end{cases}\label{eq: Deg4LBdd}
\end{equation}

To be precise, we choose $\mathfrak{g}=\mathfrak{so}(5,\mathbb{C})$
and develop ${\cal L}_{4}(V)$ into a Cartan subalgebra according
to Example \ref{exa: SO}. A Cartan subalgebra $\mathfrak{h}$ is
generated by the two elements 
\[
H_{1}=\left(\begin{array}{ccccc}
0 & 1 & 0 & 0 & 0\\
-1 & 0 & 0 & 0 & 0\\
0 & 0 & 0 & 0 & 0\\
0 & 0 & 0 & 0 & 0\\
0 & 0 & 0 & 0 & 0
\end{array}\right),\ H_{2}=\left(\begin{array}{ccccc}
0 & 0 & 0 & 0 & 0\\
0 & 0 & 0 & 0 & 0\\
0 & 0 & 0 & 1 & 0\\
0 & 0 & -1 & 0 & 0\\
0 & 0 & 0 & 0 & 0
\end{array}\right).
\]
The generators of the three root spaces $\mathfrak{g}^{\alpha},\mathfrak{g}^{\beta},\mathfrak{g}^{\gamma}$
corresponding to the specified roots $\alpha,\beta,\gamma$ in that
example can be chosen as 
\[
X_{\alpha}=\left(\begin{array}{ccccc}
0 & 0 & 0 & 0 & 0\\
0 & 0 & 0 & 0 & 0\\
0 & 0 & 0 & 0 & 1\\
0 & 0 & 0 & 0 & i\\
0 & 0 & -1 & -i & 0
\end{array}\right),
\]
\[
X_{\beta}=\left(\begin{array}{ccccc}
0 & 0 & -1 & i & 0\\
0 & 0 & i & 1 & 0\\
1 & -i & 0 & 0 & 0\\
-i & -1 & 0 & 0 & 0\\
0 & 0 & 0 & 0 & 0
\end{array}\right),\ X_{\gamma}=\left(\begin{array}{ccccc}
0 & 0 & -1 & i & 0\\
0 & 0 & -i & -1 & 0\\
1 & i & 0 & 0 & 0\\
-i & 1 & 0 & 0 & 0\\
0 & 0 & 0 & 0 & 0
\end{array}\right)
\]
respectively. We refer the reader to \cite{Helgason78},
Chapter III, Section 8 for an explicit description of the root space
decomposition of $\mathfrak{g},$ from which one will see how the
above matrices arise naturally.

Now we define $F:V\rightarrow\mathfrak{g}$ by 
\[
F({\rm e}_{1})=a_{1}X_{\alpha}+a_{2}X_{\beta}+a_{3}X_{\gamma},\ F({\rm e}_{2})=b_{1}X_{\alpha}+b_{2}X_{\beta}+b_{3}X_{\gamma},
\]
where $a_{i},b_{j}$'s are parameters to be chosen. According to Example
\ref{exa: SO}, we have $F({\cal L}_{4}(V))\subseteq\mathfrak{h}.$
We choose $\rho:\mathfrak{g}\rightarrow\mathbb{C}^{5}$ to be the
canonical matrix representation, where $\mathbb{C}^{5}$ is equipped
with the standard Hermitian norm. A common eigenbasis of $\mathbb{C}^{5}$
for all elements in $\mathfrak{h}$ under $\rho$ is given by 
\[
w_{1}=\varepsilon_{5},\ w_{2}=i\varepsilon_{1}+\varepsilon_{2},\ w_{3}=-i\varepsilon_{1}+\varepsilon_{2},\ w_{4}=i\varepsilon_{3}+\varepsilon_{4},\ w_{5}=-i\varepsilon_{3}+\varepsilon_{4},
\]
where $\{\varepsilon_{1},\cdots,\varepsilon_{5}\}$ is the canonical
basis of $\mathbb{C}^{5}.$ For $H=xH_{1}+yH_{2}\in\mathfrak{h},$
the set of eigenvalues of $\rho(H)$ with respect to the above eigenbasis
(listed in the same order) is $\{0,-ix,ix,-iy,iy\}.$ Denote $\mu$
as the weight defined by $H=xH_{1}+yH_{2}\mapsto iy,$ the eigenvalue
with respect to the common eigenvector $w_{5}.$

Suppose that $c_{1},c_{2},c_{3}>0$, under which we have $\|\pi_{4}(l)\|=8c_{1}+8c_{2}+6c_{3}.$
We then set up a polynomial system 
\begin{equation}
\mu(F(h_{1}))=8,\ \mu(F(h_{2}))=8,\ \mu(F(h_{3}))=6.\label{eq: Deg4Syst}
\end{equation}
The left hand side consists of homogeneous polynomials of degree $4$
in six variables $a_{i},b_{j}$. To simplify computation, we restrict
ourselves to solutions satisfying $a_{2}=a_{3},b_{2}=b_{3}.$ Under
this constraint, by explicit calculation it is seen that $\pm\mu(F(h_{i}))$
become the only possibly nonzero eigenvalues of $\Phi(h_{i})$ ($i=1,2,3$),
and the system (\ref{eq: Deg4Syst}) reads
\[
\begin{cases}
4a_{1}a_{3}(a_{1}b_{3}-a_{3}b_{1})=1,\\
4b_{1}b_{3}(a_{1}b_{3}-a_{3}b_{1})=-1,\\
8(a_{1}^{2}b_{3}^{2}-a_{3}^{2}b_{1}^{2})=3.
\end{cases}
\]
Treating $a_{1}$ as a free variable, the above system can be solved
explicitly to yield precisely four scenarios:
\[
\begin{cases}
a_{3}=\pm\frac{\sqrt{10}}{10a_{1}},\\
b_{1}=-\frac{a_{1}}{2},\\
b_{3}=\pm\frac{\sqrt{10}}{5a_{1}},
\end{cases},\ \ \ \begin{cases}
a_{3}=\pm\frac{\sqrt{10}i}{10a_{1}},\\
b_{1}=2a_{1},\\
b_{3}=\mp\frac{\sqrt{10}}{20a_{1}}.
\end{cases}
\]
In other words, the solution set $\Sigma\subseteq\mathbb{C}^{4}$
has complex dimension one and consists of four irreducible components
$\Sigma_{1},\Sigma_{2},\Sigma_{3},\Sigma_{4},$ each being globally
parametrized by $a_{1}\in\mathbb{C}\backslash\{0\}.$

Finally, we try to minimize the operator norm of $\Phi$ over $\Sigma$.
To this end, first recall that given an $n\times n$ complex matrix
$A$, when viewed as a linear transformation over $\mathbb{C}^{n}$, the
operator norm of $A$ with respect to the standard Hermitian norm
on $\mathbb{C}^{n}$ coincides with the maximal singular value of
$A$. By direct calculation, on the component $\Sigma_{1},$ the sets
of singular values of $\Phi({\rm e}_{1})$, $\Phi({\rm e}_{2})\in{\rm End}(\mathbb{C}^{5})$
are 
\[
\left\{ 0,\sqrt{2}|a_{1}|,\frac{2\sqrt{5}}{5|a_{1}|}\right\} ,\ \left\{ 0,\frac{|a_{1}|}{\sqrt{2}},\frac{4\sqrt{5}}{5|a_{1}|}\right\} 
\]
respectively. Therefore, we have 
\[
\|\Phi\|_{\mathbb{R}^2\rightarrow\mathrm{End}(\mathbb{C}^5)}=\max\left\{ \|\Phi({\rm e}_{1})\|_{\mathbb{C}^5\rightarrow\mathbb{C}^5},\|\Phi({\rm e}_{2})\|_{\mathbb{C}^5\rightarrow\mathbb{C}^5}\right\} =\max\left\{ \sqrt{2}|a_{1}|,\frac{4\sqrt{5}}{5|a_{1}|}\right\} .
\]
It is now elementary to see that the minimum of $\|\Phi\|_{\mathbb{R}^2\rightarrow\mathrm{End}(\mathbb{C}^5)}$
over $\Sigma_{1}$ is achieved at $|a_{1}|=2\cdot10^{-1/4}$, and
the minimum equals $2\sqrt{2}\cdot10^{-1/4}.$ Similar calculation
over the other three components of $\Sigma$ yields exactly the same
minimum. Therefore, we conclude that 
\[
\inf_{\Sigma}\|\Phi\|_{\mathbb{R}^2\rightarrow\mathrm{End}(\mathbb{C}^5)}=2\sqrt{2}\cdot10^{-1/4},
\]
and the infimum is achieved at a Lie algebraic development determined
by, for instance, 
\[
a_{1}=2\cdot10^{-1/4},\ a_{2}=a_{3}=\frac{1}{2}\cdot10^{-1/4},\ b_{1}=-10^{-1/4},\ b_{2}=b_{3}=10^{-1/4}.
\]
Under this development, we have the lower bound 
\[
L_{4}({\bf X})\geqslant\frac{8c_{1}+8c_{2}+6c_{3}}{\|\Phi\|_{\mathbb{R}^2\rightarrow\mathrm{End}(\mathbb{C}^5)}^{4}}=\frac{5}{32}\|\pi_{4}(l)\|.
\]

The discussion for other sign conditions on the coefficients $c_{1},c_{2},c_{3}$
is entirely analogous by adjusting the signs on the right hand side of the system
(\ref{eq: Deg4Syst}) accordingly. This eventually leads us
to precisely two scenarios of the desired lower estimate (\ref{eq: Deg4LBdd}).
We omit the lengthy and repeating calculations.

On the other hand, if one of the coefficients $c_{1},c_{2},c_{3}$ is zero, the lower
bound can be improved further, since one equation from the system
(\ref{eq: Deg4Syst}) is removed which produces a higher dimensional
solution set. Indeed, when $c_{3}=0$, one obtains the sharp lower bound and hence Conjecture
\ref{conj: LengConjPure} holds for this case. A simple choice of Lie algebraic developments
achieving the sharp lower bound is the following. Choose $\mathfrak{g}$ to be $\mathfrak{sl}(4,\mathbb{C})$,
the representation $\rho$ to be the canonical matrix representation,
and the embedding $F:V\rightarrow\mathfrak{g}$ to be given by
\[
{\rm e}_{1}\mapsto A\triangleq\left(\begin{array}{cccc}
0 & 1 & 0 & 0\\
0 & 0 & 1 & 0\\
0 & 0 & 0 & 1\\
-1 & 0 & 0 & 0
\end{array}\right),\ {\rm e}_{2}\mapsto B\triangleq\left(\begin{array}{cccc}
0 & 1 & 0 & 0\\
0 & 0 & -1 & 0\\
0 & 0 & 0 & 1\\
1 & 0 & 0 & 0
\end{array}\right)
\]
if $c_{1}\cdot c_{2}\geqslant0,$ and 
\[
{\rm e}_{1}\mapsto{\rm e}^{\frac{5\pi i}{8}}\cdot A,\ {\rm e}_{2}\mapsto{\rm e}^{\frac{\pi i}{8}}\cdot B
\]
if $c_{1}\cdot c_{2}<0$, respectively. The same conclusion is true
in degree $5$ when $\pi_{5}(l)$ consists of a single Hall polynomial.
We again omit the similar type of calculations.

\section{\label{sec: FreeLie}The case of Hilbert-Schmidt tensor norm:
proof of Theorem \ref{thm: FreeLie}}

As we mentioned earlier (cf. Theorem
\ref{thm: FreeLie}), Conjecture \ref{conj: LengConjPure} can be
proved for a special class of pure rough paths if we work with the
Hilbert-Schmidt tensor norm instead. Here we give an independent proof of this result. 

Let $V=\mathbb{R}^{d}$ be equipped with the $l^{2}$-metric with
respect to the standard basis $\{{\rm e}_{1},\cdots,{\rm e}_{d}\}$.
We equip each $V^{\otimes m}$ with the $l^{2}$-metric with respect
to the standard tensor basis. They extend to an inner product structure
$\langle\cdot,\cdot\rangle$ on the subalgebra $T(V)$ of $T((V))$ consisting
of finite tensors by requiring that $V^{\otimes m}$ and $V^{\otimes n}$
are orthogonal if $m\neq n$. By considering basis elements and using
bilinearity, it is immediate that 
\[
\langle\xi_{m}\otimes\xi_{n},\eta_{m}\otimes\eta_{n}\rangle=\langle\xi_{m},\eta_{m}\rangle\cdot\langle\xi_{n},\eta_{n}\rangle
\]
for all $\xi_{m},\eta_{m}\in V^{\otimes m}$ and $\xi_{n},\eta_{n}\in V^{\otimes n}.$

Recall from the assumption that ${\bf X}_{t}=\exp(t(l_{a}+l_{b}))\in G^{(b)}(V)$,
where $a<b$ and  $l_{a}$, $l_{b}$ are homogeneous Lie polynomials of degrees
$a$, $b$ respectively. Suppose that $(b-a)/{\rm gcd}(a,b)$
is an odd integer. We aim at showing that $L_{b}({\bf X})=\|l_{b}\|$.

For each $k\geqslant1,$ we write 
\begin{equation}
\pi_{bk}\left(\exp(l_{a}+l_{b})\right)=\frac{l_{b}^{\otimes k}}{k!}+Q,\label{eq: Expansion}
\end{equation}
where the exponential is now taken over $T((V)),$ and $Q$ is sum
of all remaining terms in the expansion. The key step is to show that,
if $(b-a)/{\rm gcd}(a,b)$ is odd, then $l_{b}^{\otimes k}$ and $Q$
are orthogonal for all large $k$. This can be proved by making use
of an anti-automorphism on the tensor algebra together with symmetry
properties of the signature expansion. The orthogonality property
clearly leads to the lower bound 
\[
\|\pi_{bk}(\exp(l_{a}+l_{b}))\|\geqslant\frac{\|l_{b}\|^{k}}{k!}.
\]
Combining with the general upper bound given by Theorem \ref{thm: upper},
the result then follows.

To prove (\ref{eq: Expansion}), first consider the linear map $\alpha:T(V)\rightarrow T(V)$
induced by 
\[
\alpha({\rm e}_{i_{1}}\otimes\cdots\otimes{\rm e}_{i_{m}})=(-1)^{m}{\rm e}_{i_{m}}\otimes\cdots\otimes{\rm e}_{i_{1}}.
\]
By definition, $\alpha$ is an anti-involution, i.e. $\alpha(\xi\otimes\eta)=\alpha(\eta)\otimes\alpha(\xi)$
and $\alpha^{2}={\rm Id}$. In addition, for any $\xi,\eta\in T(V),$
we have $\langle\alpha(\xi),\alpha(\eta)\rangle=\langle\xi,\eta\rangle$.
A crucial property of $\alpha$ is that $\alpha(l)=-l$ for any Lie
polynomial $l$. The notion of $\alpha$ and the above properties can
be found in \cite{Reutenauer93}, Chapter 1. An immediate
consequence of using the anti-involution $\alpha$ is the following
lemma. Recall that the \textit{symmetrized product} of $\xi_{1},\cdots,\xi_{n}\in T(V)$
is defined by 
\[
{\rm Sym}(\xi_{1},\cdots,\xi_{n})\triangleq\frac{1}{n!}\sum_{\sigma\in{\cal S}_{n}}\xi_{\sigma(1)}\otimes\cdots\otimes\xi_{\sigma(n)},
\]
where ${\cal S}_{n}$ is the permutation group of order $n$. For
convenience, we also define the \textit{reduced symmetrized product
}
\[
{\rm R{\rm Sym}}(\underbrace{\xi_{1},\cdots,\xi_{1}}_{k_{1}\ {\rm times}},\cdots,\underbrace{\xi_{n},\cdots,\xi_{n}}_{k_{n}\ {\rm times}})\triangleq\frac{1}{k_{1}!\cdots k_{n}!}{\rm Sym}(\underbrace{\xi_{1},\cdots,\xi_{1}}_{k_{1}\ {\rm times}},\cdots,\underbrace{\xi_{n},\cdots,\xi_{n}}_{k_{n}\ {\rm times}}).
\]

\begin{lem}
\label{lem: Orth}Let $l_{0},l_{1},\cdots,l_{n}$ be Lie polynomials
and $k\geqslant1.$ If $k+n$ is an odd integer, then 
\[
\langle l_{0}^{\otimes k},{\rm Sym}(l_{1},\cdots,l_{n})\rangle=0.
\]
The same result  holds for the reduced symmetrized product.
\end{lem}
\begin{proof}
Observe that 
\[
\alpha\left({\rm Sym}(l_{1},\cdots,l_{n})\right)=(-1)^{n}{\rm Sym}(l_{1},\cdots,l_{n}).
\]
Therefore, we have 
\begin{align*}
\langle l_{0}^{\otimes k},{\rm Sym}(l_{1},\cdots,l_{n})\rangle & =\langle\alpha(l_{0}^{\otimes k}),\alpha({\rm Sym}(l_{1},\cdots,l_{n}))\rangle\\
 & =(-1)^{k+n}\langle l_{0}^{\otimes k},{\rm Sym}(l_{1},\cdots,l_{n})\rangle.
\end{align*}
The first assertion follows since $k+n$ is odd by assumption. The
second assertion is obvious.
\end{proof}
Now we are in a position to give the proof of Theorem \ref{thm: FreeLie}.

\begin{proof}[Proof of Theorem \ref{thm: FreeLie}]

We express the remainder $Q$ in the expression (\ref{eq: Expansion})
in a more explicit way:
\begin{equation}
Q=\sum_{x>0,\ ax+by=bk}{\rm RSym}(\underbrace{l_{a},\cdots,l_{a}}_{x\ {\rm times}},\underbrace{l_{b},\cdots,l_{b}}_{y\ {\rm times}}).\label{eq: Q}
\end{equation}
An important observation is that, for each summand, since $ax+by=bk,$
we have
\[
\frac{a}{r}x=\frac{b}{r}(k-y)
\]
where $r\triangleq{\rm gcd}(a,b),$ showing that $b/r\mid x$ and
thus $x\geqslant b/r.$ For the equation to make sense, one also needs
$k\geqslant a/r.$

Firstly, if $x=b/r$, then $y=k-a/r$. In this case, we have 
\[
k+x+y=2k+\frac{b-a}{r},
\]
which is an odd integer by assumption. According to Lemma \ref{lem: Orth},
we conclude that 
\begin{equation}
\langle l_{b}^{\otimes k},{\rm RSym}(\underbrace{l_{a},\cdots,l_{a}}_{x\ {\rm times}},\underbrace{l_{b},\cdots,l_{b}}_{y\ {\rm times}})\rangle=0.\label{eq: Orth}
\end{equation}

Next, consider a given $x>b/r$ from the sum in (\ref{eq: Q}). For
each single term $\xi$ in the corresponding reduced symmetrized product,
$\xi$ can be uniquely written as $\xi=\xi_{1}\otimes\xi_{2},$ where
$\xi_{1}$ contains exactly $b/r$ number of $l_{a}$'s and $\xi_{2}$
starts with $l_{a}.$ Let $S$ be the set of all such $\xi_{2}$'s
arising in this way. Denote $y(\xi_{2})$ as the number of $l_{b}$'s
in each given $\xi_{2}\in S.$ Then the reduced symmetrized product
can further be written as 
\begin{align*} 
 & {\rm RSym}(\underbrace{l_{a},\cdots,l_{a}}_{x\ {\rm times}},\underbrace{l_{b},\cdots,l_{b}}_{y\ {\rm times}})\\
 & =\sum_{\xi_{2}\in S}\left(\frac{b}{r}+y-y(\xi_{2})!\right)\cdot{\rm RSym}(\underbrace{l_{a},\cdots,l_{a}}_{b/r\ {\rm times}},\underbrace{l_{b},\cdots,l_{b}}_{y-y(\xi_{2})\ {\rm times}})\otimes\xi_{2}. 
\end{align*}
For each $\xi_{2}\in S,$ by writing $k_{1}\triangleq a/r+y-y(\xi_{2}),$
Lemma \ref{lem: Orth} again implies that 
\begin{align*}
 %\langle l_{b}^{\otimes k},{\rm RSym}(\underbrace{l_{a},\cdots,l_{a}}_{x\ {\rm times}},\underbrace{l_{b},\cdots,l_{b}}_{y\ {\rm times}})\rangle
 \langle l_{b}^{\otimes k_{1}},{\rm RSym}(\underbrace{l_{a},\cdots,l_{a}}_{b/r\ {\rm times}},\underbrace{l_{b},\cdots,l_{b}}_{y-y(\xi_{2})\ {\rm times}})\rangle\cdot\langle l_{b}^{\otimes(k-k_{1})},\xi_{2}\rangle=0,
\end{align*}
since 
\[
k_{1}+\frac{b}{r}+y-y(\xi_{2})=2k_{1}+\frac{b-a}{r}
\]
is an odd integer. Therefore, (\ref{eq: Orth}) holds for the reduced
symmetrized product corresponding to the given $x$.

It follows that $l_{b}^{\otimes k}$ is orthogonal to $Q$ provided
$k\geqslant a/r,$ and the proof of the theorem is now complete.

\end{proof}

\section*{Acknowledgement}

HB and NS are supported by EPSRC grant EP/R008205/1. XG is  supported in part by NSF grant DMS1814147.

\section*{Appendix: Some properties of pure rough paths}\label{sec:Appendix}
\addcontentsline{toc}{section}{\nameref{sec:Appendix}}

In this section, we prove the two properties of pure rough paths stated
in Proposition \ref{prop: LocalPVar}
and Proposition \ref{prop: SigPure} respectively in Section \ref{subsec: Pure}.

\begin{proof}[Proof of Proposition \ref{prop: LocalPVar}]

Let ${\bf X}_{t}=\exp(tl)$ ($0\leqslant t\leqslant1$) be a pure
$m$-rough path, where $l\in{\cal L}^{(m)}(V)$ with $l_{m}\triangleq\pi_{m}(l)\neq0$.
For each $1\leqslant k\leqslant m,$ the degree $k$ component of
${\bf X}_{s,t}\triangleq{\bf X}_{s}^{-1}\otimes{\bf X}_{t}$ has the
form
\begin{align*}
X_{s,t}^{k} & =\sum_{r=0}^{m}\frac{(t-s)^{r}}{r!}\pi_{k}(l^{\otimes r})=\sum_{r=1}^{k}(t-s)^{r}\xi_{r}^{(k)},
\end{align*}
where $\xi_{r}^{(k)}\in V^{\otimes k}$ are tensors constructed from
$\pi_{1}(l),\cdots,\pi_{m}(l)$. It follows that 
\begin{equation}
\|X_{s,t}^{k}\|\leqslant\sum_{r=1}^{k}|t-s|^{r}\|\xi_{r}^{(k)}\|\leqslant C_{{\bf X}}\cdot|t-s|,\label{eq: EstX^k}
\end{equation}
where $C_{{\bf X}}$ denotes a constant depending only on ${\bf X}$. This implies
that ${\bf X}$ is an $m$-rough path in the sense of Definition \ref{def: RoughPath}.

Now if $k<m,$ from (\ref{eq: EstX^k}) we have
\[
\|X_{s,t}^{k}\|^{\frac{m}{k}}\leqslant C_{{\bf X}}^{\frac{m}{k}}\cdot|t-s|^{\frac{m}{k}},
\]
showing that 
\[
\lim_{{\rm mesh({\cal P})\rightarrow0}}\sum_{t_{i}\in{\cal P}}\|X_{t_{i-1},t_{i}}^{k}\|^{\frac{m}{k}}=0.
\]
If $k=m,$ notice that $\xi_{1}^{(m)}=l_{m}.$ Therefore, given a
finite partition ${\cal P}$ of $[0,1],$ we have
\[
\|X_{t_{i-1},t_{i}}^{m}\|=(t_{i}-t_{i-1})\cdot\left\|l_{m}+(t_{i}-t_{i-1})\xi_{2}^{(m)}+\cdots+(t_{i}-t_{i-1})^{m-1}\xi_{m}^{(m)}\right\|.
\]
It is now elementary to see that 
\[
\lim_{{\rm mesh}({\cal P})\rightarrow0}\sum_{t_{i}\in{\cal P}}\|X_{t_{i-1},t_{i}}^{m}\|=\|l_{m}\|.
\]
Therefore, we conclude that the local $m$-variation of ${\bf X}$
equals $\|l_{m}\|.$

\end{proof}
\begin{rem}
This property apparently extends to the non-geometric setting, i.e.
for the case when $l\in T^{(m)}(V).$ Indeed, even more holds true with essentially the same proof. Let $\mathbf{X}_t=\exp(L(t))\in T^{(m)}(V)$, where $L(t)$ is a bounded variation path in $T^{(m)}(V)$. Then \[
\lim_{\delta\rightarrow0}\sum_{k=1}^{m}\left(\inf_{{\rm mesh}({\cal P})\leqslant\delta}\sum_{t_{i}\in{\cal P}}\|X_{t_{i-1},t_{i}}^{k}\|^{\frac{m}{k}}\right)^{\frac{k}{m}}=\|\pi_{m}(L)\|_{1\text{-var}}.
\]
\end{rem}
\begin{proof}[Proof of Proposition \ref{prop: SigPure}]

Let $\mathbf{X}_{t}=\exp(tl)\in G^{(m)}(V)$ be a pure $m$-rough
path. For any $n\geqslant m,$ it is not hard to see that the multiplicative
functional ${\bf X}_{s,t}^{(n)}\triangleq\exp((t-s)l)\in T^{(n)}(V)$
has finite total $m$-variation, where the exponential is taken over
$T^{(n)}(V).$ Therefore, ${\bf X}^{(n)}$ is the unique extension
of ${\bf X}$ to $T^{(n)}(V)$ given by Theorem \ref{thm: LyonsExt}.
By the definition of signature, $\exp(l)$ is the signature of ${\bf X}$
where the exponential is now taken over $T((V)).$ The second part of
the proposition is a direct consequence of the uniqueness result for
signature in \cite{BGLY16}.

\end{proof}

\end{document}